\documentclass[12pt,reqno,a4paper]{amsart}
\usepackage[utf8]{inputenc}
\usepackage{graphicx, amsmath, amssymb, amscd, amsthm, euscript, psfrag, amsfonts,bm}

\usepackage[colorlinks]{hyperref}
\hypersetup{
linkcolor=blue,          
citecolor=green,        
}

\usepackage{cleveref}
\usepackage{etoolbox}
\usepackage{amsmath}
\usepackage{enumerate}
\usepackage{amssymb}
\usepackage{amscd}
\usepackage{amsthm}
\usepackage{amsfonts}
\usepackage{graphicx}
\usepackage[all,cmtip]{xy}

\usepackage{todonotes}

\patchcmd{\subsection}{-.5em}{.5em}{}{}
\patchcmd{\subsubsection}{-.5em}{.5em}{}{}

\usepackage{enumitem}
\crefformat{section}{\S#2#1#3}
\crefformat{subsection}{\S#2#1#3}
\crefformat{subsubsection}{\S#2#1#3}

\newtheorem{theorem}{Theorem}[section]
\newtheorem{lemma}[theorem]{Lemma}
\newtheorem{proposition}[theorem]{Proposition}
\newtheorem{corollary}[theorem]{Corollary}
\newtheorem{claim}{Claim}[theorem]
\newtheorem*{genericthm*}{\thistheoremname}
\newenvironment{namedthm*}[1]
  {\renewcommand{\thistheoremname}{#1}
   \begin{genericthm*}}
  {\end{genericthm*}}
\theoremstyle{definition}
\newtheorem{definition}[theorem]{Definition}
\newtheorem{example}[theorem]{Example}

\theoremstyle{remark}
\newtheorem{remark}[theorem]{Remark}

\newcommand{\bs}{\boldsymbol}
\newcommand\numberthis{\addtocounter{equation}{1}\tag{\theequation}}
\newcommand{\bR}{\mathbb{R}}
\newcommand{\bx}{B}
\newcommand{\cube}{\mathcal{C}}
\newcommand{\good}{(C,\alpha){\text -} good}
\newcommand{\vx}{\boldsymbol{x}}

\newcommand{\vy}{\boldsymbol{y}}
\newcommand{\valpha}{\boldsymbol{\alpha}}
\newcommand{\classP}{\mathcal{P}_l(\mathbb{R}^k,G)}
\newcommand{\prm}{^{\prime}}
\numberwithin{equation}{section}

\title[Locally unipotent invariant measures] 
      {Locally unipotent invariant measures and limit distribution of a sequence of polynomial trajectories on homogeneous spaces} 

\author[Han Zhang]{}


 \email{hanzhang3906@tsinghua.edu.cn}
 
\begin{document}
\maketitle

\centerline{\scshape Han Zhang}
\medskip
{\footnotesize
 \centerline{Yau Mathematical Sciences Center}
   \centerline{Tsinghua University}
   \centerline{Beijing, 100084, China}
} 

\bigskip


\begin{abstract}
Let $G$ be a Lie group and $\Gamma$ be a lattice in $G$. We introduce the notion of locally unipotent invariant measures on $G/\Gamma$. We then prove that under some conditions, the limit measure supported on the image of polynomial trajectories on $G/\Gamma$ is locally unipotent invariant, thus give a partial answer to an equidistribution problem for higher dimensional polynomial trajectories on homogeneous spaces, which was raised by Shah in \cite{shah1994limit}.

The proof relies on Ratner's measure classification theorem, linearization technique for polynomial trajectories near singular sets and a twisting technique of Shah.
\end{abstract}


\section{Introduction}
Let $G$ be a Lie group and $\Gamma$ be a lattice in $G$, that is, a discrete subgroup of $G$ such that $X=G/\Gamma$ admits a $G$-invariant probability measure. Let $U=\{u(t):t\in\bR\}$ be a one-parameter unipotent subgroup of $G$. Ratner's uniform distribution theorem \cite{ratner1991raghunathanduke} asserts that for any $x\in G/\Gamma$, there exists a closed connected subgroup $F$ of $G$ containing $U$ such that $Fx$ is closed and admit an $F-$invariant probability measure $\mu_F$ so that $Ux$ is uniformly distributed on $Fx$ with respect to $\mu_F$, i.e., for any $f\in C_c(X)$,
\begin{align*}
    \lim_{T\to \infty}\frac{1}{T}\int_0^T f(u(t)x)dt=\int_{Fx}f d\mu_{F}.
\end{align*}

The key property of unipotent subgroup used in the proof of Ratner's uniform distribution theorem is that the map $t \mapsto Ad u(t)$ is a polynomial function in each coordinate of $End(Lie(G))$, where $Ad$ is the adjoint representation of $G$ on its Lie algebra. The property of polynomial growth rate enables Shah to prove a limit distribution result for polynomial trajectories in the following setting \cite{shah1994limit}.

Let $G$ be a real algebraic group, that is, $G$ is an open subgroup of $\bR$-points of an algebraic group $\bs{G}$ defined over $\bR$. A map $\Theta:\bR^k \rightarrow G$ is called regular algebraic if it is the restriction of a morphism $\Theta: \mathbb{C}^k \rightarrow \bs{G}$ of algebraic varieties defined over $\bR$ (cf. \cite[\S 1]{shah1994limit}). If $G$ is a subgroup of $SL_N(\bR)$ for some $N\in \mathbb{N}$, then each matrix coordinate of $\Theta$ is a polynomial in $\bR^k$.

We say that $\Theta$ is of \textit{product type} if we can write
\begin{align}\label{product type}
    \Theta(\vx)=\Theta_k(x_k)\cdots\Theta_1(x_1),\forall \vx=(x_1,\cdots,x_k)\in \bR^k,
\end{align}
where $\Theta_i:\bR \to G$ is a one-dimensional regular algebraic map for $1\leq i\leq k$. The main theorem  in \cite{shah1994limit} asserts the following:

\begin{theorem}(Shah)\label{Main theorem of Shah}
Let $G$ be a real algebraic group. Let $\Gamma\subset G_1\subset G$ be closed subgroups of $G$ such that $G_1/\Gamma$ admits a $G_1$-invariant probability measure. Assume that $\Theta:\bR^k\to G_1$ is a regular algebraic map of \textbf{product type} and $\Theta(\boldsymbol{0})=e$. Then there exists a smallest closed subgroup F of $G_1$ containing $\Theta(\bR^k)$ such that the orbit $F\Gamma/\Gamma$ is closed in $G/\Gamma$ and admits an $F$-invariant probability measure $\mu_F$ such that the following holds: given any sequences $T_n^{(1)}\to \infty,\cdots,T_n^{(k)}\to \infty$ as $n\to \infty$, for the boxes $\bx_n=[0,T_n^{(1)}]\times\cdots \times [0,T_n^{(k)}]$, we have for any $f\in C_c(G/\Gamma)$,
\begin{align}\label{limit distribution of product type}
    \lim_{n\to \infty} \frac{1}{|B_n|}\int_{\vx\in B_n} f(\Theta(\vx)\Gamma)d\vx=\int_{F\Gamma/\Gamma}f d\mu_F,
\end{align}
where $|\cdot|$ denotes the Lebesgue measure on $\bR^k$.
\end{theorem}

Theorem \ref{Main theorem of Shah} has many applications, for example see \cite{shah2000counting}\cite{shah2018expanding}. In \cite{shah1994limit}, a natural question was raised about the limit distribution of $\Theta$ \textit{without} product type assumption. Shah answered this question under the assumption that the averaging is taken on regular balls $B_n$ centered at $\boldsymbol{0}$ of $\bR^k$ (cf. \cite[Corollary 1.2]{shah1994limit}). When the averaging is taken on arbitrary expanding boxes, that is, $B_n$ is an arbitrary box as in Theorem \ref{Main theorem of Shah}, the question remains open.

The assumption that $\Theta$ is of product type make it possible for one to study the limit distribution of $\Theta$ using the result on one-dimensional polynomial trajectories on homogeneous spaces. Without product type assumption, we are forced to deal with the higher dimensional polynomial trajectories. Because additional dimensions introduce complexity, we put some regularity conditions on the boxes $B_n$ and the map $\Theta$. The following is our main result.

\begin{theorem}\label{main theorem}
Let $G$ be a real algebraic group. Let $\Gamma\subset G_1\subset G$ be closed subgroups such that $G_1/\Gamma$ admits a $G_1$-invariant probability measure. Let $\Theta:\bR^k \to G_1$ be a nonconstant regular algebraic map \textbf{without} product type assumption. Assume that $\Theta(\bs{0})=e$. Let $F$ be the smallest closed subgroup of $G_1$ containing $\Theta(\mathbb{R}^k)$ such that the orbit $F\Gamma/\Gamma$ is closed. Further, assume that $F$ is semisimple and algebraic. Then $F\Gamma/\Gamma$ admits an $F$-invariant probability measure $\mu_F$ such that the following holds: given any sequence $T_n \to \infty$ as $n\to \infty$ and $\lambda_i>0$ for $i=1,\cdots,k$, for the boxes $B_n=[0,T_n^{\lambda_1}]\times\cdots\times[0,T_n^{\lambda_k}]$,  
\begin{align}\label{equation:limit measure of measures supported on Bn}
     \lim_{n\to \infty} \frac{1}{|B_n|}\int_{\vx\in B_n} f(\Theta(\vx)\Gamma)d\vx=\int_{F\Gamma/\Gamma} f d\mu_F,
\end{align}
for any $f\in C_c(G/\Gamma)$.
\end{theorem}

\begin{remark}
By \cite[Proposition 2.1]{shah1994limit}, the smallest closed subgroup $F$ of $G$ containing $\Theta(\bR^k)$ is generated by algebraic one-parameter unipotent subgroups, thus $F$ is an algebraic group. By \cite[Lemma 2.9]{shah1991uniformly}, $F$ admits a Levi decomposition $F=L\cdot N$, where $L$ is a semisimple algebraic group without compact factors and $N$ is the unipotent radical of $F$. In the above theorem, we assume that $N$ is trivial and $F\Gamma/\Gamma$ is closed.
\end{remark}

Here we provide an easy example where our main Theorem \ref{main theorem} applies but Theorem \ref{Main theorem of Shah} does not.
\begin{example}
Let $G=SL_2(\bR)$, $\Gamma=SL_2(\mathbb{Z})$ and $\mu_G$ be the $G$-invariant probability measure on $G/\Gamma$. Consider a regular algebraic map $\Theta:\bR^k \to G$ for some $k\geq 1$ by
\begin{align*}
    \Theta(\vx)=\begin{pmatrix}
    \Theta_{11}(\vx) & \Theta_{12}(\vx)\\
    \Theta_{21}(\vx) & \Theta_{22}(\vx)
    \end{pmatrix}, \forall \vx\in \bR^k.
\end{align*}
Assume that none of $\Theta_{ij}$ is constant for $1\leq i,j\leq 2$. Then it can be verified that there is no algebraic subgroup $F$ of $G$ containing $\Theta(\bR^k)$. Theorem \ref{main theorem} implies that for any sequence $T_n\to \infty$ as $n\to \infty$ and any positive $\lambda_i$ for $1\leq i\leq k$, we have 
\begin{align*}
      \lim_{n\to \infty} \frac{1}{|B_n|}\int_{\vx\in B_n} f(\Theta(\vx)\Gamma)d\vx=\int_{G/\Gamma} f d\mu_G,
\end{align*}
for any $f\in C_c(G/\Gamma)$, where $B_n=[0,T_n^{\lambda_1}]\times\cdots\times[0,T_n^{\lambda_k}]$.
\end{example}

One of the main difficulties of proving the limit distribution result in Theorem \ref{main theorem} is to obtain unipotent invariance property of limit measure. Unlike \cite[Section 4]{shah1994limit}, we are not able to get unipotent invariance without product type assumption of $\Theta$.

The proof of Theorem \ref{main theorem} relies on the fact that the limit measure is locally unipotent invariant, which we will give precise definition in section \ref{section: locally unipotent invariance}. Roughly speaking, a probability measure $\mu$ on $G/\Gamma$ is locally unipotent invariant if there exists a sequence of probability measures $\{\mu_i\}$ converging to $\mu$ such that one can find another sequence of probability measures $\{\nu_i\}$ converging to a probability measure $\nu$ and the following hold: (1) each $\nu_i$ is a suitable perturbation of $\mu_i$; (2) $\nu$ is invariant under a nontrivial unipotent subgroup of $G$. Therefore, a locally unipotent invariant measure $\mu$ is invariant under a nontrivial unipotent subgroup up to small perturbations. The twisting technique of Shah (cf. \cite{shah2009equidistribution}) enables one to study locally unipotent invariant measures by Ratner's theorems.

In the setting of Theorem \ref{main theorem}, given a sequence $T_n\to \infty$ as $n\to \infty$ and positive real numbers $\lambda_i$ for $i=1,\cdots,k$, we can write
\begin{align}\label{rescale boxes}
    B_n=\{(\alpha_1 T_n^{\lambda_1},\cdots,\alpha_k T_n^{\lambda_k}): \valpha= (\alpha_1,\cdots,\alpha_k)\in [0,1]^k\}.
\end{align}
Given a regular algebraic map $\Theta: \mathbb{R}^k \rightarrow G$ such that $\Theta(\boldsymbol{0})=e$, we define a map $\theta: [0,1]^k\times (0,\infty) \rightarrow G$ as follows: 
\begin{align}\label{definition of rescaling map}
    \theta(\valpha, t)=\Theta(\alpha_1 t^{\lambda_1},\cdots,\alpha_k t^{\lambda_k}),
\end{align}
where $\valpha\in [0,1]^k$ and $t\in (0,\infty)$.

By the above change of variables, to study the limit distribution of $\Theta$ on $B_n$ as $n\to \infty$, it is equivalent to study the limit distribution of the sequence of polynomial trajectories $\{\theta_n=\theta(\cdot,T_n)\}_{n\in\mathbb{N}}$ on $[0,1]^k$ as $n\to \infty$.

Let $J$ be a box contained in $[0,1]^k$. Let
\begin{align}\label{expanding subbox}
    B_{n}^J=\{(\alpha_1 T_n^{\lambda_1},\cdots,\alpha_k T_n^{\lambda_k}):\valpha= (\alpha_1,\cdots,\alpha_k)\in J\}.
\end{align}

For each $B_n^J$, we define the probability measure $\mu_n^J=\mu_{B_n}^J$ as follows: for any $f\in C_c(G/\Gamma)$
\begin{align}\label{measure over subbox}
    \int_{G/\Gamma} f d\mu_{n}^J=\frac{1}{|J|}\int_J f(\theta(\valpha, T_n))d\valpha.
\end{align}
We will show that for any nontrivial box $J\subset [0,1]^k$, the measure $\mu_n^J$ constructed above converges to a locally unipotent invariant probability measure $\mu^J$ on $G/\Gamma$. Other examples of locally unipotent invariant measures had been studied in \cite{shah2009equidistribution} and \cite{yang2020equidistribution}.

We fix some notations and conventions in this paper.

(1) By a box in $\bR^k$ we always mean a box with faces parallel to the coordinate planes, same for cubes. We denote cubes by $\mathcal{C}$. All boxes and cubes are assumed to be closed if unspecified.

(2) For vectors in $\mathbb{R}^k$, we will denote them by boldface letters $\vx$, $\valpha$, etc, to distinguish them from real numbers.

(3) We let $dist(\cdot,\cdot)$ denote the right invariant Riemannian metric on $G$. $d_X$ or $d_{G/\Gamma}$ will denote the induced metric on $G/\Gamma$.

(4) We use the absolute value symbol $|\cdot|$ to denote the usual absolute value of a real number as well as the Lebesgue measure of a measurable set in $\bR^k$. This should not cause any confusion.

The paper is organized as follows: In section \ref{Nondivergence of the limit distribution of polynomial trajectories}, we prove that the limit distribution of $\Theta$ on expanding boxes is nondivergent. In section \ref{trajectories near singular set}, we study higher dimensional polynomial trajectories near singular sets of homogeneous spaces by incorporating $\good$ property of polynomials and Besicovitch's covering theorem into linearization technique. We also give a careful treatment of twisting technique (see Corollary \ref{modification of linearization}). In section \ref{section: locally unipotent invariance}, we introduce the notion of locally unipotent invariant measures. Assuming Lemma \ref{local unipotent invariance}, we show that the limit measure of measures defined as in (\ref{equation:limit measure of measures supported on Bn}) is locally unipotent invariant. In section \ref{proof of main theorems}, we prove the main theorem by adapting the twisting technique and applying Ratner's measure classification theorem. In the last section \ref{Section Invariance under a unipotent flow}, we present the proof of Lemma \ref{local unipotent invariance}.

\section{Nondivergence of the limit distribution of polynomial trajectories}\label{Nondivergence of the limit distribution of polynomial trajectories}

Let $G$ be a Lie group. For $k,l\in \mathbb{N}$, let $\classP$ denote the set of continuous maps $\Theta:\bR^k\rightarrow G$ such that for all $\bs{a},\bs{b}\in \bR^k$ and $v\in Lie(G)$, the map 
\begin{align*}
    t\in \bR \mapsto Ad\circ \Theta(t\bs{a}+\bs{b})(v)\in Lie(G)
\end{align*}
is a polynomial of degree at most $l$ in each coordinate of $Lie(G)$.

The following is a nondivergence theorem for polynomial trajectories.
\begin{theorem}\label{Nondivergence}(\cite[Theorem 3.1]{shah1994limit}) Let G be a Lie group and $\Gamma$ be a closed subgroup such that $G/\Gamma$ admits a finite G-invariant measure. Then, given a compact set $C\subset G/\Gamma$, an $\epsilon >0$, and an $l\in \mathbb{N}$, there exists a compact subset $K\subset G/\Gamma$ with the following property: for any $g\Gamma\in G/\Gamma$ and $\Theta \in \classP$, and any bounded open convex set $B\subset \bR^k$, if $\Theta(B)g\Gamma\cap C \neq \emptyset$, then 
\begin{align*}
    \frac{1}{|B|}|\{\vx\in B: \Theta(\vx)g\Gamma\in K\}|> 1-\epsilon.
\end{align*}

\end{theorem}

Let $\mathcal{P}(G/\Gamma)$ be the space of Borel probability measures on $G/\Gamma$ with the weak-* topology. Let $\Theta:\bR^k\to G$ be a regular algebraic map such that $\Theta(\boldsymbol{0})=e$. Then $\Theta\in \mathcal{P}_l(\bR^k,G)$ for some $l$. For any bounded open convex set $B \subset \bR^k$, we define $\mu_{B}\in \mathcal{P}(G/\Gamma)$ such that for any $f\in C_c(G/\Gamma)$,
\begin{align}\label{normalized meausre over box}
    \int_{G/\Gamma} f d\mu_{B}=\frac{1}{|B|}\int_{\vx\in B}f(\Theta(\vx)\Gamma)d\vx.
\end{align}

We have the following corollary:
\begin{corollary}(\cite[Corollary 3.1]{shah1994limit})\label{Cor of nondivergence of whole box}
Given a sequence $\{B_n\}_{n\in \mathbb{N}}$ of bounded open convex subsets of $\bR^k$ containing $\bs{0}$, there exists a strictly increasing subsequence $\{n_i\}_{i\in \mathbb{N}}$ and a measure $\mu\in \mathcal{P}(G/\Gamma)$ such that $\mu_{B_{n_i}}\to \mu$ as $i\to \infty$.
\end{corollary}

Let $J$ be a box contained in $[0,1]^k$. Let $B_n$ and $B_n^J$ be defined as in (\ref{rescale boxes}) and (\ref{expanding subbox}), respectively. Recall that we also define the measure $\mu_n^J$ as in (\ref{measure over subbox}).

An immediate consequence of Theorem \ref{Nondivergence} and Corollary \ref{Cor of nondivergence of whole box} is the following:

\begin{corollary}\label{cor nondivergence over subbox}
Assume that $|J|=\epsilon>0$, then there exists a strictly increasing subsequence $\{n_i\}_{i\in\mathbb{N}}$ and a measure $\mu^J\in \mathcal{P}(G/\Gamma)$ such that $\mu_{n_i}^J\to \mu^J$ as $i \to \infty$.
\end{corollary}

\begin{proof}
Consider the one-point compactification $X^*$ of $X=G/\Gamma$. By the compactness of $\mathcal{P}(X^*)$, after passing to a subsequence, we can assume that $\mu^J_n$ converges to a probability measure $\mu^J$ in $\mathcal{P}(X^*)$ as $n\to\infty$. Again by passing to a subsequence, we can assume that $\mu_n=\mu_{B_{n}}\to \mu$ as $n\to \infty$, for some $\mu\in \mathcal{P}(X)$ by Corollary \ref{Cor of nondivergence of whole box}.

Suppose that $\mu^J(X^*\backslash X)=\epsilon\prm >0$. By Theorem \ref{Nondivergence}, there exists a compact $K\subset X$ such that for all $n$ large enough, $\mu_n(K)>1-\epsilon\prm \epsilon /2$. On the other hand, since $\mu^J(X^*\backslash X)=\epsilon\prm >0$, for all $n$ large enough, $\mu_n^J(K)\leq 1-\epsilon\prm/2$. Therefore for all $n$ large enough, we have the following estimates:
\begin{align*}
    \mu_n(K)=&\frac{1}{|B_{n}|}(\int_{B_n\backslash B^{J}_n}\chi_K(\Theta(\vx)\Gamma)d\vx+\int_{B^{J}_n} \chi_K(\Theta(\vx)\Gamma)d\vx)\\
    & \leq \frac{1}{|B_n|}(|B_n\backslash B^{J}_n|+|B^{J}_n|(1-\epsilon\prm/2))\\
    &\leq 1-\epsilon+\epsilon (1-\epsilon\prm/2)=1-\epsilon \epsilon\prm/2,
\end{align*}

where $\chi_K$ is the characteristic function of $K$. But $\mu_n(K)>1-\epsilon\prm \epsilon /2$, this lead to a contradiction and completes the proof.
\end{proof}

\section{Linearization and polynomial trajectories near singular sets}\label{trajectories near singular set}
\begin{definition}\label{good function}
Given $C>0$, $\alpha>0$, and a box $B \subset \bR^k$, a continuous function $f:\bR^d\rightarrow \mathbb{R}$ is $\good$ on $B$ if one has
\begin{equation}
    \forall \delta>0, |\{\vx \in B:|f(\vx)|<\delta\}|\leq C\left(\frac{\delta}{\|f\|_B}\right)^{\alpha}|B|,
\end{equation}
where $\|f\|_\bx=sup_{\vx\in B}|f(\vx)|$.
\end{definition}

The notion of $\good$ functions, which measures the growth property of a continuous function, was first introduced in \cite{kleinbock1998flows}. Throughout this section, by degree of a multi-variable polynomial, we mean the highest degree among all variables, for example, degree of $x^5y^2$ is 5.

\begin{lemma}\label{polynomial is good}
Let $P:\bR^k \rightarrow \bR$ be a polynomial map of degree $\leq l$. Then for any box $\bx \subset \bR^k$, P is $\good$ on $\bx$, where $\alpha=\frac{1}{kl}$ and $C=C_{k,l}\geq 1$ is a constant only depending on $k,l$.
\end{lemma}

\begin{proof}
This can be deduced directly from the proof of Lemma 3.3 in \cite{kleinbock1998flows} using the induction hypothesis. See also \S 4 of \cite{kleinbock2008dirichlet}.
\end{proof}

The following lemma is an easy consequence of a map being $\good$.
\begin{lemma}\label{consequence of good}
Let $P:\bR^k\to \bR$ be a continuous map. Suppose $P$ is $\good$ on a measurable set $E$. Assume $\sup_{\vx\in E\prm}|P(\vx)|<R$ for some measurable subset $E\prm$ of $E$ with positive Lebesgue measure and some $R>0$. Then there is an $R\prm>0$ such that $\sup_{\vx\in E}|P(\vx)|<R\prm$.
\end{lemma}
\begin{proof}
By $\good$ property of the function $P$, we have
\begin{align*}
  | \{\vx\in E: |P(\vx)|<R\}|<C(\frac{R}{\sup_{\vx\in E}|P(\vx)|})^{\alpha}\cdot|E|.
\end{align*}
On the other hand, by assumption, we also have 
\begin{align*}
     |E\prm|<|\{\vx\in E: |P(\vx)|<R\}|.
\end{align*}
Take $R\prm=\frac{C^{1/\alpha} R|E|^{1/\alpha}}{|E\prm|^{1/\alpha}}$, the lemma is proved.
\end{proof}

In the rest of this section, we assume that $G$ is a Lie group and $\Gamma$ is a closed subgroup of $G$ such that $\Gamma^0$, the connected component of $\Gamma$, is normal in $G$. We will use linearization technique to study the behavior of polynomial trajectories near singular sets. The linearization technique originated in the work of Dani and Smillie \cite{dani1984uniform}. This technique had been developed by several mathematicians, see for example \cite{dani1990orbit}\cite{shah1991uniformly}\cite{dani1993limit}\cite{lindenstrauss2019quantitative}. The recent study of linearization technique in \cite{lindenstrauss2019quantitative} also gives a quantitative description of unipotent orbits near singular sets.

The following special class of closed subgroups plays a key role in linearization technique.

\begin{definition}\label{Ratner's Class}
Let $\mathcal{H}$ be the class of all closed connected proper subgroups $H$ of $G$ such that $\Gamma^0\subset H$, $H/H\cap\Gamma$ admits an $H$-invariant probability measure and the subgroup generated by all unipotent one-parameter subgroups of $H$ acts ergodically on $H/H\cap\Gamma$ with respect to the $H$-invariant probability measure. 
\end{definition}

\begin{theorem}(cf. \cite[Theorem 1.1]{ratner1991raghunathanannals})
The collection $\mathcal{H}$ is countable.
\end{theorem}

Let $\pi:G\to G/\Gamma$ be the natural projection map. Let $W$ be a subgroup of $G$ which is generated by one-parameter unipotent subgroups of $G$ contained in $W$. For $H\in\mathcal{H}$, define 
\begin{align*}
    N(H,W)=\{g\in G: W\subset gHg^{-1}\};\\
    S(H,W)=\bigcup_{F\in \mathcal{H},F\subsetneq H}N(F,W).
\end{align*}

For a proper subgroup $H\in \mathcal{H}$, $N(H,W)$ is the obstruction of equidistribution of $W$-orbit in $G/\Gamma$. We call such $N(H,W)$ a singular set.

The following is a version of Ratner's theorem describing probability measures invariant under the subgroup $W$ given as above.

\begin{theorem}(cf. \cite[Theorem 2.2]{mozes1995space}) \label{Ratner's Theorem}
Let $W$ be a subgroup which is generated by one-parameter unipotent subgroups of $G$ contained in $W$. Let $\mu\in \mathcal{P}(G/\Gamma)$ be a $W$-invariant measure. There exists $H\in \mathcal{H}$ such that 
\begin{align*}
    \mu(\pi(N(H,W)))>0\text{ and }\mu(\pi(S(H,W)))=0.
\end{align*}
Moreover, almost every $W$-ergodic component of $\mu$ on $\pi(N(H,W))$ is a measure of the form $g_{*}\mu_H$, where $g\in N(H,W)\setminus S(H,W)$ and $\mu_H$ is a finite $H$-invariant measure on $\pi(H)$. In particular, if $H$ is a normal subgroup of $G$ then $\mu$ is invariant under $H$.
\end{theorem}

Let $H\in\mathcal{H}$. Let $\mathfrak{g},\mathfrak{h}$ be the Lie algebras of $G$ and $H$, respectively. Let $d=\dim \mathfrak{h}$ and $V_H=\bigwedge^d \mathfrak{g}$. Consider the adjoint representation of $G$ on $V_H=\bigwedge^d \mathfrak{g}$. Fix a vector $p_H\in \bigwedge^d \mathfrak{h}\setminus \{0\}$. Also define a continuous map $\eta_H: G\to V_H$ by $\eta_H(g)=g\cdot p_H=\bigwedge^d Ad(g)p_H$. Define 
\begin{align}\label{definition of $N^1(H)$}
    N^1(H):=\eta_H^{-1}(p_H)=\{g\in N(H): det(Ad g|_{\mathfrak{h}})=1\},
\end{align}
where $N(H)$ is the normalizer of $H$ in $G$.

Recall that $W$ is a subgroup of $G$ which is generated by unipotent one-parameter subgroups of $G$ contained in $W$. Let $A_H$ denote the Zariski closure of $\eta_H(N(H,W))$ in ${V}_H$.
Evidently, $N(H,W)$ is contained in the preimage of $A_H$. Indeed, we have the following:
\begin{lemma}(\cite[Proposition 5.1]{shah1994limit})\label{preimage of zariski closure of N(H,W)}
Let $H\in \mathcal{H}$, then $\eta_H^{-1}(A_H)=N(H,W).$
\end{lemma}

\begin{theorem}\label{discrete orbit}(\cite[Theorem 5.1]{shah1994limit})
Let $H\in \mathcal{H}$, then

(1) The orbit $\Gamma\cdot p_H$ is closed, hence discrete;

(2) $N^1(H)\Gamma$ is closed in $G/\Gamma$.
\end{theorem}

Recall that throughout this section, we assume that $\Gamma$ is a closed subgroup and $\Gamma^0$ is normal in $G$ instead of $\Gamma$ being a discrete subgroup in $G$. The following theorems stated in \cite{dani1993limit} still hold because proof of them essentially only use the properties of subgroups in the class $\mathcal{H}$ (see Definition \ref{Ratner's Class}) and that $\eta_H(\Gamma)$ is discrete for $H\in \mathcal{H}$ as in Theorem \ref{discrete orbit}.

Let $\Gamma_H=N(H)\cap \Gamma$. For any $\delta\in \Gamma_H$, $\delta$ preserves the volume of $H\Gamma/\Gamma$. Hence, $\delta\cdot p_H=\pm p_H$. Indeed, $\Gamma_H$ is characterized by this property.

\begin{lemma}\label{lemma stablizer in Gamma H}(\cite[Lemma 3.1]{dani1993limit})
Let $H\in \mathcal{H}$, then $\Gamma_H=\{\gamma\in \Gamma: \eta_H(\gamma)=\pm p_H\}$.
\end{lemma}

Let $H\in \mathcal{H}$ and $Y$ be a subset of $G$. Let $\pi_Y$ be the natural quotient map $\pi_Y:Y\Gamma_H/\Gamma_H \to Y\Gamma/\Gamma$. Following \cite{dani1993limit}, $y\in Y$ is said to be a point of $(H,\Gamma)$-self intersection of $Y$ if $\pi_Y^{-1}(y\Gamma)$ has more than one point, that is, there exists $\gamma \in \Gamma-\Gamma_H$ such that $y\gamma \in Y$.

\begin{proposition}(\cite[Proposition 3.3]{dani1993limit})\label{proposition self intersection points are in proper subvarieties}
Let $H\in \mathcal{H}$, then the set of $(H,\Gamma)$-self intersection points of $N(H,W)$ is contained in $S(H,W)$.
\end{proposition}

\begin{corollary}(\cite[Corollary 3.5]{dani1993limit})\label{cor no self intersection points}
Let $H\in \mathcal{H}$ and let $A_H$ be the Zariski closure of $\eta_H(N(H,W))$ in $V_H$. Let $D$ be a compact subset of $A_H$. Let $Y_H$ be the set of $(H,\Gamma)$-self intersection points of $\eta_H^{-1}(D)$ and let $K_1$ be a compact subset of $G-Y_H\Gamma$. Then there exists a neighborhood $\Phi$ of $D$ in $V_H$ such that the quotient map from $(\eta_H^{-1}(\Phi)\Gamma_H \cap K_1\Gamma)/\Gamma_H$ onto $(\eta_H^{-1}(\Phi)\Gamma\cap K_1\Gamma)/\Gamma$ is injective.
\end{corollary}

We also need the following proposition describing the intersection of compact subsets of singular sets:

\begin{proposition}(cf. \cite[Proposition 7.2]{dani1993limit})\label{proposition intersection of compact subsets of singular subvarieties}
Let $H\in \mathcal{H}$ and $D$ be a compact subset of $A_H$. Let $K$ be a compact subset of $G/\Gamma$. Then $\{K\cap (\eta_H^{-1}(D)\cap \eta_H^{-1}(D)\gamma)\Gamma/\Gamma\}_{\gamma\in\Gamma}$ is a family of sets with only finitely many distinct elements. Furthermore, for each $\gamma\in \Gamma$, there exists a compact subset $C_{\gamma}$ of $\eta_H^{-1}(D)\cap \eta_H^{-1}(D)\gamma$ such that 
\begin{align*}
    K\cap (\eta_H^{-1}(D)\cap \eta_H^{-1}(D)\gamma)\Gamma/\Gamma=C_{\gamma}\Gamma/\Gamma
\end{align*}
\end{proposition}

The following proposition can be used to study higher dimensional polynomials trajectories near singular sets.

\begin{proposition}\label{behavior of polynomial}
Let $H\in \mathcal{H}$. Let a compact set $C\subset A_H$ and an $\epsilon>0$ be given. Then there exists a compact set $D\subset A_H$ such that $C\subset D$, and for any neighborhood $\Phi$ of $D$ in ${V}_H$, there exists a neighborhood $\Psi$ of $C$ in ${V}_H$ with $\Psi \subset \Phi$ such that for any $\Theta\in \classP$ and any box $B \subset \bR^k$, if there exists $\boldsymbol{a} \in  B$ and some $v_0\in {V}_H$ with $\Theta(\boldsymbol{a})v_0 \notin \Phi$, then
\begin{equation}\label{relative size}
    |\{\vx \in B:\Theta(\vx)v_0 \in \Psi\}|\leq \epsilon|\{\vx\in B:\Theta(\vx)v_0 \in \Phi\}|
\end{equation}
\end{proposition}

\begin{remark}\label{remark relative size}
Given a compact set $C\subset A_H$ and $\epsilon>0$, let $D$ be the compact set obtained in Proposition \ref{behavior of polynomial} containing $C$. Following \cite{dani1993limit}, we say that $C$ is of relative size less than $\epsilon$ in $D$.
\end{remark}

Here we incorporate $\good$ property with Besicovitch's covering theorem to prove the above proposition.

\begin{theorem}(\cite[Theorem 2.7]{mattila1999geometry})\label{besicovich}
There is an integer $N_k$ depending only on $k$ with the following property: let  $A$ be a bounded subset of $\bR^k$ and let $C$ be a family  of nonempty cubes in $\bR^k$ such that each $\vx\in A$ is the center of some cube belongs to $C$; 
then there exists a finite or countable subfamily $\{\cube(i)\}$ of C with $
1_A \le \sum_i 1_{\cube(i)} \le N_{k}$.
\end{theorem}

Note that the above covering theorem is stated for balls in \cite[Theorem 2.7]{mattila1999geometry}, but it also holds if balls are replaced by cubes.

\begin{proof}[proof of proposition \ref{behavior of polynomial}]
Our proof is similar to that of \cite[Proposition 4.2]{dani1993limit}. Since the proposition is trivial if $\epsilon \geq 1$, we can assume $\epsilon <1$. Denote 
\begin{align*}
    A_H=\{v\in {V}_H: P(v)=0\},
\end{align*}
where $P$ is a polynomial in some fixed basis of  ${V}_H$ of degree less than $m$. Fix some norm $\|\cdot\|$ on ${V}_H$. Let $r>0$ be such that $C\subset B_r(0) \subset {V}_H$, where $B_r(0)$ is the ball centered at origin with radius $r$. Let $\alpha=\frac{1}{2ml}$, c=$C_{2ml,k}$, where $C_{2ml,k}$ is as defined in Lemma \ref{polynomial is good}. Let $\delta=(c^{-1} \epsilon N_k^{-1})^{\frac{1}{\alpha}}$, where $N_k$ is the constant defined in Besicovitch's covering theorem. Define a compact set
\begin{align*}
    D=\{v\in A_H : \|v\|\leq \frac{r}{\sqrt{\delta}}\}.
\end{align*}
Let $\Phi$ be a neighborhood of $D$ in ${V}_H$, then there exists $\beta>0$ such that 
\begin{equation*}
 \{v\in {V}_H: \|v\|<\frac{r+\beta}{\sqrt{\delta}},|P(v)|<\beta\}\subset \Phi.    
\end{equation*}
Now define
\begin{equation*}
     \Psi=\{v\in {V}_H: \|v\|<r+\beta, |P(v)|<\beta \delta\},
\end{equation*}
then $\Psi \subset \Phi$. Consider functions $\varphi_i(\vx):\bR^k \rightarrow \bR$ for $i=1,2$ defined by
\begin{equation*}
\varphi_1(\vx)=\frac{\delta\|\Theta(\vx)v_0\|^2}{(r+\beta)^2},\varphi_2(\vx)=\frac{P(\Theta(\vx)v_0)}{\beta}.
\end{equation*}
Note that for $i=1,2$, deg($\varphi_i$)$\leq 2ml$. Let $\varphi=max(|\varphi_1|,|\varphi_2|)$, then $\varphi$ is $\good$ on any box by \cite[lemma 3.1]{kleinbock1998flows} and $\varphi(\boldsymbol{a})\geq 1$.
Note that we have 
\begin{equation}\label{inequality 1}
    \{\vx \in \bx: \Theta(\vx)v_0\in \Psi\}\subset \{\vx \in \bx: \varphi(\vx)<\delta\},
\end{equation}
and
\begin{equation}\label{inequality 2}
   \{\vx \in \bx: \varphi(\vx)<1\} \subset \{\vx \in \bx: \Theta(\vx)v_0\in \Phi\}.
\end{equation}

Let $E=\{\vx\in B:\varphi(\vx)\geq 1\}$, then $E\neq \emptyset$ since $\boldsymbol{a}\in E$. We construct a Besicovitch's covering for $B\setminus E$ as follows: for each $\vx \in B\setminus E$, let $\cube(\vx)$ be the smallest cube centered at $\vx$ such that $sup_{\cube(\vx)\cap B}|\varphi|=1$. Note that $\cube(\vx)$ is nonempty. Then we have $B\setminus E \subset \bigcup_{\vx\in B \setminus E}\cube(\vx)$. Thus, by Besicovitch's covering theorem, there is a countable subcollection $\{\cube(i): i \in \mathbb{N}\}$ of $\{\cube(\vx):\vx \in B \setminus E\}$ with multiplicity at most $N_k$ and $B \setminus E =\bigcup_i \cube(i)\cap B$. Note that $\cube(i)\cap B$ is a box for any $i$. Therefore by $\good$ property of $\varphi$, we have
\begin{align*}
|\{\vx\in B:\varphi(\vx)<\delta\}|&\leq \sum_i|\{\vx\in B\cap \cube(i):\varphi(\vx)<\delta\}|\\
     &\leq \sum_i c\delta^{\alpha}|B\cap \cube(i)| \\
     &\leq c\delta^{\alpha}N_k|\bigcup_i B\cap \cube(i)|\\
     &=\epsilon|B\setminus E|=\epsilon |\{\vx \in B: \varphi(\vx)<1\}|.
\end{align*}
Where the second inequality follows from $\good$ property of $\varphi$ and choice of $\cube(i)$, the third inequality follows from Besicovitch's covering theorem. Combining (\ref{inequality 1}) and (\ref{inequality 2}), this proves the proposition.
\end{proof}

Now we study the dynamics of polynomial trajectories in thin neighborhoods of singular sets. For finitely many subgroups $H_1,\cdots,H_r\in \mathcal{H}$ and compact subsets $D_1,\cdots,D_r$ of $A_{H_1},\cdots,A_{H_r}$ respectively, consider the neighborhood $\Phi_i$ of $D_i$ in ${V}_{H_i}$ for $1\leq i \leq r$. Following \cite[\S 7]{dani1993limit}, by component of $\eta^{-1}_{H_1}(\Phi_1)\cup \cdots \cup \eta^{-1}_{H_r}(\Phi_r)$ we mean $\eta^{-1}_{H_i}(\Phi_i)$ for some $1\leq i \leq r$. 
\begin{theorem}\label{result of linearization of polynomial map}
Let $H\in \mathcal{H}$ and $W$ be a nontrivial closed connected subgroup of $G$ which is generated by Ad-unipotent elements contained in it. Given a compact subset $C\subset N(H,W)$, an $\epsilon>0$ and a compact subset $K\subset G/\Gamma$, there exist $H_i \in \mathcal{H}$ with  $H_i\subset H$ and $dim H_i< dim H$ for $1\leq i \leq r$ for some $r\in \mathbb{N}$, such that the following is satisfied: Denote $H_0=H$, then we can find compact subsets $D_i\subset A_{H_i}$ for $0\leq i\leq r$ such that given any neighborhood $\Phi_i$ of $D_i$ in ${V}_{H_i}$ for $0\leq i \leq r$, there exists a neighborhood $\Omega$ of $\pi(C)$ in $G/\Gamma$ such that for any $\Theta \in \classP$, any $g\in G$ and any box $B \subset \bR^k$, either\\
$(1)$ $\exists \gamma \in \Gamma$, such that $\Theta(B)g\gamma$ is contained in one of the components of $\eta^{-1}_{H_0}(\Phi_0)\cup \cdots \cup \eta^{-1}_{H_r}(\Phi_r)$, or\\
$(2)$ $|\{\vx \in B: \Theta(\vx)g\Gamma\in \Omega \cap K\}|\leq \epsilon |B|$.
\end{theorem}

The above theorem can be viewed as a higher dimensional variant of \cite[Theorem 7.3]{dani1993limit}. Unlike the proof in \cite{Eskin_Mozes_Shah_Unipotent_flows_and_counting_1996} where one-dimensional version of the above theorem is applied, here we give an alternative proof of Theorem \ref{result of linearization of polynomial map} by Besicovitch's covering type arguments, which does not rely on the one-dimensional version result.

\begin{proof}
We use induction on $dim H$. If $dim H$ is small $(dim H<dim W)$, then $N(H,W)=\emptyset$ and the theorem is trivial. 
Now assume $dim H=n$ and the theorem is proved for $F\in \mathcal{H}$ with $dim F\leq n-1$.

Let $C^{\prime}=\eta_H(C)$. According to Proposition \ref{behavior of polynomial}, choose $D$ to be a compact subset of ${V}_H$ such that $C\prm$ is of relative size less than $\epsilon N_k/2^{k+2}$ in $D$ (see Remark \ref{remark relative size}), where $N_k$ is the multiplicity in Besicovitch's covering theorem. By Proposition \ref{proposition intersection of compact subsets of singular subvarieties}, the set $\{K\cap(\eta_H^{-1}(D)\cap \eta_H^{-1}(D)\gamma^{-1})\Gamma/\Gamma\}_{\gamma \in \Gamma}$ has only finitely many elements, say $K\cap(\eta_H^{-1}(D)\cap \eta_H^{-1}(D)\gamma_0^{-1})\Gamma/\Gamma,\cdots, K\cap(\eta_H^{-1}(D)\cap \eta_H^{-1}(D)\gamma_m^{-1})\Gamma/\Gamma$, where $\gamma_0=e$. Moreover, for $j\geq 1$, $K\cap(\eta_H^{-1}(D)\cap \eta_H^{-1}(D)\gamma_j^{-1})\Gamma/\Gamma$ is of the form $C_j\Gamma/\Gamma$ where $C_j$ is a compact subset of $N(F_j,W)$ with $F_j\subset H$ and $\dim F_j<\dim H$.

Applying the induction hypothesis to $C_j$, $N(F_j,W)$ and $\frac{\epsilon}{2m}$ in place of $C$, $N(H,W)$ and $\epsilon$ for $1\leq j \leq m$, then there exist $H_1,\cdots,H_r \subset H$ with $\dim H_i <\dim H$ (more precisely, each $H_i$ is a subgroup of some $F_j$) and compact set $D_i \subset A_{H_i}$ for $1\leq i \leq r$ such that given any neighborhood $\Phi_i$ of $D_i$ for $1\leq i \leq r$, there exists a neighborhood $\Omega_j$ of $C_j\Gamma/\Gamma$ for $1\leq j\leq m$ such that if we define $\Omega^{\prime}=\bigcup_{j=1}^{m}\Omega_j$, then for any $\Theta\in \classP$, any $g\in G$ and any box $B \subset \bR^k$, either

(i) $\exists \gamma \in \Gamma$, such that $\Theta(B)g\gamma$ is contained in one of the components of $\eta^{-1}_{H_1}(\Phi_1)\cup \cdots \cup \eta^{-1}_{H_r}(\Phi_r)$,or
 
(ii) $|\{\vx \in B: \Theta(\vx)g\Gamma\in \Omega^{\prime} \cap K\}|\leq \frac{\epsilon}{2}| B|$.
 
Now let $K_1=K\setminus\Omega^{\prime}$ and choose a compact set $K_1\prm\subset G$ such that $K_1\prm \Gamma/\Gamma=K_1 \subset G/\Gamma$. Given a neighborhood $\Phi_0$ of $D$ in ${V}_H$, applying Corollary \ref{cor no self intersection points}, we can find a neighborhood $\Phi$ of $D$ in ${V}_H$ such that $\Phi\subset \Phi_0$ and the natural quotient map
\begin{equation}\label{self intersection}
    (\eta_H^{-1}(\Phi)\Gamma_H\cap K_1\prm \Gamma)/\Gamma_H\rightarrow (\eta_H^{-1}(\Phi)\Gamma\cap K_1\prm \Gamma)/\Gamma
\end{equation}
is injective.

Now choose a neighborhood $\Psi$ in ${V}_H$ of $C\prm$ according to Proposition \ref{behavior of polynomial} such that if there exist $v_0\in {V}_H$ and $\boldsymbol{a}\in {\bx}$ with $\Theta(\boldsymbol{a})v_0 \notin \Phi$, then 
\begin{equation}\label{choice of neighborhoods}
    |\{\vx \in \bx:\Theta(\vx)v_0 \in \Psi\}|\leq \frac{\epsilon}{2^{k+2}N_k}|\{\vx\in \bx:\Theta(\vx)v_0 \in \Phi\}|.
\end{equation}

We will show that the theorem holds for $\Omega=\eta^{-1}_H(\Psi)\Gamma/\Gamma$.
Suppose (1) in the theorem does not hold, so in particular, (i) does not hold either. Then for any $\gamma \in \Gamma$, there exists $\boldsymbol{a}\in \bx$ such that $\Theta(\boldsymbol{a})g\gamma p_H \notin \Phi$. For each $q\in \eta_H(\Gamma)$, we define 
\begin{equation*}
    I(q)=\{\vx\in \bx: \Theta(\vx)gq\in \Phi\};
\end{equation*}
\begin{equation*}
    J(q)=\{\vx\in \bx: \Theta(\vx)gq\in \Psi, \Theta(\vx)g\Gamma \in K_1\};
\end{equation*}
\begin{equation*}
    S(q)=\{\vx \in I(q): \Theta(\vx)g\Gamma \in K_1\}.
\end{equation*}

We construct a Besicovitch's covering for $S(q)$ as follows: For each $\vx\in S(q)$, let $\cube^q(\vx)$ be the largest cube in $\bR^k$ centered at $\vx$ such that $\Theta(\cube^q(\vx)\cap \bx)g q \subset \Phi$. Note that $\cube^q(\vx)$ has nonempty interior and is bounded by assumption. So $S(q)\subset \bigcup_{\vx\in S(q)}\cube^q(\vx)$.

As $S(q)$ is bounded, and each $\cube^q(\vx)$ is nonempty by construction, $\{\cube^q(\vx):\vx\in S(q)\}$ is a Besicovitch's covering. By Besicovitch's covering theorem, we can extract a countable subcover $\{\cube^q(j):j\in \mathbb{N}\}$ with $N_k$ multiplicity, such that $S(q) \subset\bigcup^{\infty}_{j=1}\cube^q(j)$. We then define 
\begin{align*}
  I\prm(q)=\bigcup_{j=1}^{\infty}\cube^q(j)\cap \bx  .
\end{align*}
Clearly by the above constructions, we have $J(q)\subset S(q)\subset I\prm(q)\subset I(q)$ for all $q\in \eta_H(\Gamma)$.

Consider a finite set
\begin{equation*}
  Q=\{\boldsymbol{s}_1,\cdots, \boldsymbol{s}_{2^k}:\boldsymbol{s}_i\in \mathbb{R}^k \text{ and each coordinate of $\boldsymbol{s}_i$ is either 1 or -1}\} . 
\end{equation*}
We write $\boldsymbol{s}_i=(s_{i1},\cdots,s_{ik})$ with $s_{ij}$ being $1$ or $-1$ for $1\leq j\leq k$. We use the set $Q$ to index $2^k$ different quadrants of $\mathbb{R}^k$. For $\vx\in \bR^k, t\geq 0$ and $1\leq i\leq 2^k$, define a cube in the direction of $i$-th quadrant of $\vx$ with size $t$ by
\begin{align*}
     \cube^i_t(\vx)=[x_1,x_1+t\cdot s_1]\times\cdots\times [x_k,x_k+t\cdot s_k].
\end{align*}

Furthermore, for each $1\leq i \leq 2^k$ and $q\in \eta_H(\Gamma)$, define 
\begin{align*}
    Q^i(q)=\{\vx \in I(q): \exists t\geq 0 \text{ such that } \cube^i_t(\vx)\cap \bx \subset I(q) \text{ and }\\ \exists \boldsymbol{y} \in \cube^i_t(\vx)\cap \bx, \text{ such that } \Theta(\boldsymbol{y})g\Gamma \in K_1 \}.
\end{align*}
We will use the set $Q^i(q)$ to study the intersection property of sets $I\prm(q)$ for different $q$'s. Indeed, we claim the following:
\begin{claim}\label{injectivity} 
For any $1\leq i\leq 2^k$, and $q_1,q_2\in \eta_H(\Gamma)$, if $Q^i(q_1)\cap Q^i(q_2)\neq \emptyset$, then $q_1=\pm q_2$.
\end{claim}
\begin{proof}[proof of the claim]
By assumption, there exists $\vx\in Q^i(q_1)\cap Q^i(q_2)$ and $t_j \geq 0$ such that there exists $\vy_j \in \cube_{t_j}^i(\vx)\cap \bx \subset I(q_j)$ and $\Theta(\vy_j)g\Gamma \in K_1$ for $j=1,2$. Without loss of generality, we can assume $t_1\leq t_2$, then $\cube_{t_1}^i(\vx)\subset \cube_{t_2}^i(\vx)$ and therefore we have $\Theta(\vy_1)g q_j \in \Phi$ for $j=1,2$. Let $q_j=\eta_H(\gamma_j)$ for $j=1,2$. Since the map (\ref{self intersection}) is injective and $\Theta(\vy_1)g\Gamma \in K_1$, we have $\gamma_1\Gamma_H=\gamma_2\Gamma_H$. By Lemma \ref{lemma stablizer in Gamma H}, $q_1=\pm q_2$.
\end{proof}

Define $I^i(q)=I\prm(q)\cap Q^i(q)$, for $1\leq i\leq 2^k$. We have the following claim:
\begin{claim}\label{breakup I(q)}
$I\prm(q)=\bigcup_{i=1}^{2^k}I^i(q)$ for any $q\in \eta_H(\Gamma)$.
\end{claim}
\begin{proof}[proof of the claim]
Fix $q \in \eta_H(\Gamma)$, let $\vx \in I\prm(q)$. By definition of $I\prm(q)$, $\vx$ is in some cube $\cube(\vy)$ centered at $\vy\in I(q)$ such that $\Theta(\vy)g\Gamma \in K_1$ and $\cube(\vy)\cap B \subset I(q)$. Therefore, depending on the quadrant where $\vy$ lies with respect to $\vx$, we can find $1\leq i \leq 2^k$ and $t\geq 0$ such that $\vy\in \cube^i_t(\vx)\cap\bx$ and $\cube^i_t(\vx)\subset \cube(\vy)$. Since $\cube^i_t(\vx)\cap\bx\subset I(q)$, by definition, $\vx \in Q^i(q)$. Therefore, $I\prm(q)\subset \bigcup_{i=1}^{2^k}Q^i(q)$.
\end{proof}

Recall that $I\prm(q)=\bigcup_{j=1}^{\infty}\cube^q(j)\cap \bx$. For each $j$, we have 
\begin{align*}
    |J(q)\cap\cube^q(j)\cap \bx |&=|\{\vx\in \cube^q(j)\cap \bx:\Theta(\vx)g q\in \Psi, \Theta(\vx)g\Gamma\in K_1 \}|\\
    &\leq \frac{\epsilon}{2^{k+2}N_k}|\{\vx \in \cube^q(j)\cap \bx: \Theta(\vx)gq\in \Phi\}|\\
    &\leq \frac{\epsilon}{2^{k+2}N_k}|\cube^q(j)\cap \bx|,
\end{align*}

where the first inequality follows from the fact that intersection of two boxes is still a box and inequality (\ref{choice of neighborhoods}). Recall that $J(q)\subset I\prm(q)$, so we have

\begin{align*}
    |J(q)|&=|\bigcup_{j=1}^{\infty}(J(q)\cap \cube^q(j)\cap \bx)|\\
    &\leq \sum_{j=1}^{\infty}|J(q)\cap \cube^q(j)\cap \bx|\\
    &\leq \frac{\epsilon}{2^{k+2}N_k}\sum_{j=1}^{\infty}|\cube^q(j)\cap \bx|\\
    &\leq \frac{\epsilon}{2^{k+2}}|I\prm(q)|,
\end{align*}

where the last inequality follows from the fact that $N_k$ is the multiplicity of the covering (see Theorem \ref{besicovich}).

Recall that we define $\Omega=\eta^{-1}_H(\Psi)\Gamma/\Gamma$. By Theorem \ref{discrete orbit}, $\eta_H(\Gamma)$ is closed and discrete. Therefore, we have 
\begin{align*}
    |\{\vx \in \bx: \Theta(\vx)g\Gamma \in \Omega\cap K_1\}|&=|\bigcup_{q\in \eta_H(\Gamma)}J(q)|\\
    &\leq \sum_{q\in \eta_H(\Gamma)}|J(q)|\leq \frac{\epsilon}{2^{k+2}} \sum_{q\in \eta_H(\Gamma)}|I\prm(q)|\\
    &\leq \frac{\epsilon}{2^{k+2}}\sum_{q\in \eta_H(\Gamma)}\sum_{i=1}^{2^k}|I^i(q)|\\
    &\leq 2\cdot \frac{\epsilon}{2^{k+2}} \sum_{i=1}^{2^k}|\bigcup_{q\in \eta_H(\Gamma)}I^i(q)|\\
    &\leq \frac{\epsilon}{2^{k+1}}\cdot 2^k|\bigcup_{q\in \eta_H(\Gamma)}I(q)|\leq \frac{\epsilon}{2}|\bx| \numberthis \label{inequality0},
\end{align*}
where the third and fourth inequalities follow from claim \ref{injectivity} and \ref{breakup I(q)}. By assumption, we have 
\begin{align}\label{assumption that (2) holds}
     |\{\vx \in \bx: \Theta(\vx)g\Gamma \in \Omega\prm\cap K\}|\leq \frac{\epsilon}{2}|B|.
\end{align}
Combine (\ref{assumption that (2) holds}), (\ref{inequality0}) and $K_1=K\setminus \Omega\prm$, we conclude that (2) in the theorem holds and complete the proof.
\end{proof}

\begin{remark}
Let $C\subset N(H,W)$, $\epsilon>0$ and compact $K\subset G/\Gamma$ be given as in Theorem \ref{result of linearization of polynomial map}. We observe that in the proof of Proposition \ref{behavior of polynomial} and Theorem \ref{result of linearization of polynomial map}, we can slightly shrink the neighborhood $\Omega$ of $\pi(C)$ and compact set $K$ so that the conclusions still hold. Therefore we are able to twist the polynomial map $\Theta$ by a continuous map with image in a small neighborhood of $G$ so that similar dichotomy in Theorem \ref{result of linearization of polynomial map} is still true. More precisely, we have the following corollary, which can be viewed as a version of twisting technique developed by Shah:
\end{remark} 

\begin{corollary}(Twisting technique)\label{modification of linearization}
Let $H$, $W$, $C$, $\epsilon>0$, $K$ be given as in Theorem \ref{result of linearization of polynomial map}. For $0\leq i\leq r$, let $H_i$, $D_i$, $\Phi_i$ and neighborhood $\Omega$ of $\pi(C)$ be as in the conclusion of Theorem \ref{result of linearization of polynomial map} such that either (1) or (2) holds. Then depending on $C$ and $K$, there exist a compact neighborhood $\mathcal{U}$ of $e$ in $G$, a compact set $\tilde{K}\subset K$ and a neighborhood $\tilde{\Omega}\subset \Omega$ of $\pi(C)$ satisfying $\mathcal{U}\tilde{K}\subset K$ and $\mathcal{U}\tilde{\Omega} \subset \Omega$, such that for any box B, any continuous map $Z: B\rightarrow \mathcal{U}$, any $\Theta\in \classP$ and any $g\in G$, either

$(1\prm)$ $\exists \gamma \in \Gamma$, such that $Z(\bx)\Theta(\bx)g\gamma$ is contained in one of the components of $\eta^{-1}_{H_0}(\mathcal{U}\cdot\Phi_0)\cup \cdots \cup \eta^{-1}_{H_r}(\mathcal{U}\cdot\Phi_r)$, or

$(2\prm)$ $|\{\vx \in \bx: Z(\vx)\Theta(\vx)g\Gamma\in \tilde{\Omega}\cap \tilde{K}\}|\leq \epsilon |\bx|$.
\end{corollary}

\begin{proof}
Recall that in the proof of Proposition \ref{behavior of polynomial}, we choose $r>0$ such that $C\prm=\eta_H(C) \subset \{v\in {V}_H:\|v\|\leq r\}$. Indeed, we can choose $0<r\prm<r$ such that $C\prm\subset \{v\in {V}_H:\|v\|\leq r\prm\}\subset \{v\in {V}_H:\|v\|\leq r\}$. Also recall that we define $D=\{v\in {V}_H:\|v\|<r/\sqrt{\delta},|P(v)|=0\}$, where $\delta=(c^{-1}\epsilon_0 N_k^{-1})^{\frac{1}{\alpha}}$ and $\epsilon_0=\epsilon N_k/2^{k+2}$, i.e., $C\prm$ is of relative size less than $\epsilon N_k/2^{k+2}$ in $D$. Given a neighborhood $\Phi_0$ of $D$ in ${V}_H$, we can choose a neighborhood $\Phi\subset \Phi_0$ of $D$ as in the proof of Theorem \ref{result of linearization of polynomial map} such that the map (\ref{self intersection}) is injective.

There exists $\beta>0$ such that $\{v\in {V}_H:\|v\|<(r+\beta)/\sqrt{\delta},|P(v)|<\beta\}\subset \Phi$ (note that when $D$ is chosen to be large, the neighborhood $\Phi$ obtained in Corollary \ref{cor no self intersection points} need to be small so that the map (\ref{self intersection}) is injective. So typically $\beta$ could be very small). Define 
\begin{align*}
    \Psi=\{v\in {V}_H: \|v\|<r+\beta,|P(v)|<\beta\cdot\delta\}.
\end{align*}
Choose some $0<\beta\prm<\beta$, define
\begin{align*}
    \Psi\prm=\{v\in {V}_H: \|v\|<r\prm+\beta\prm,|P(v)|<\beta\prm\cdot\delta\}.
\end{align*}
Note that $C\prm \subset \Psi\prm \subset \Psi$. Then we can find a sufficiently small compact neighborhood $\mathcal{U}_1$ of e in G ($\mathcal{U}_1$ could be very small as $\beta$ is small), such that $\mathcal{U}_1\cdot \Psi\prm\subset \Psi$. Also, we can find a compact set $\tilde{K}\subset K$ and a neighborhood $\mathcal{U}_2$ of $e$ in $G$ such that $\mathcal{U}_2\tilde{K}\subset K$. Let $\mathcal{U}=\mathcal{U}_1\cap \mathcal{U}_2$, and without loss of generality, we can assume $\mathcal{U}=\mathcal{U}^{-1}$. Recall that in Theorem \ref{result of linearization of polynomial map}, we define $\Omega=\eta_H^{-1}(\Psi)\Gamma/\Gamma$. Now let $\tilde{\Omega}=\eta_H^{-1}(\Psi\prm)\Gamma/\Gamma$, we have $\mathcal{U}\tilde{\Omega}\subset \Omega$.

Let $Z:B\to \mathcal{U}$ be a continuous map. If $(1\prm)$ does not hold, then (1) in Theorem \ref{result of linearization of polynomial map} does not hold neither, thus (2) holds. Therefore we have 
\begin{align*}
    |\{\vx \in \bx: Z(\vx)\Theta(\vx)g\Gamma\in \tilde{\Omega}\cap \tilde{K}\}|&\leq  |\{\vx \in \bx: \Theta(\vx)g\Gamma\in \mathcal{U}\tilde{\Omega}\cap \mathcal{U}\tilde{K}\}|\\
    &\leq |\{\vx \in \bx: \Theta(\vx)g\Gamma\in \Omega\cap K\}|\\
    &\leq \epsilon|\bx|.
\end{align*}

So $(2\prm)$ holds. This proves the corollary.
\end{proof}

\section{Locally unipotent invariant measure}\label{section: locally unipotent invariance}

Ratner's measure classification theorem classifies all unipotent invariant probability measures on homogeneous spaces. Any unipotent invariant ergodic measure is supported on a closed $H$-orbit for some $H\in \mathcal{H}$ as shown in Theorem \ref{Ratner's Theorem}. A twisting technique developed by Shah (cf. \cite{shah2009equidistribution}) and Ratner's theorem enable one to study locally unipotent invariant measures, which we will give the definition below.

Let $(Y,\mathcal{B},\lambda)$ be a Borel measurable metric space with Borel measure $\lambda$. Let $G$ be a Lie group and $\Gamma$ be a lattice of $G$. Let $B\subset Y$ be a compact subset with $\lambda(B)>0$, and $\{\theta_n\}_{n\in \mathbb{N}}$ be a sequence of continuous maps from $Y$ to $G/\Gamma$. Consider the measure $\mu_n^B$ on $G/\Gamma$ defined by 
\begin{align}\label{definition of original measure in l.u.i}
    \mu_n^B:=\frac{1}{\lambda(B)}\int_B \delta_{\theta_n(y)\Gamma/\Gamma}d\lambda(y),
\end{align}
where $\delta$ denotes the Dirac measure. We assume that $\mu_n^B$ converges to a probability measure $\mu^B$ on $G/\Gamma$ as $n\to \infty$.

\begin{definition}\label{definition: admissible partitions}
A sequence of finite partitions $\{\mathbb{P}_b=\{B^b_a:1\leq a\leq b\}\}_{b\in \mathbb{N}}$ of a compact set $B\subset Y$ is admissible if it satisfies the following:

(1) each $B^b_a$ is a closed set in $B$ with nonempty interior;

(2) for any $B^{b}_{a_0}\in \mathbb{P}_{b}$, there exists a $B^{b+1}_{a_1}\in \mathbb{P}_{b+1}$ such that $B^{b+1}_{a_1}\subset B^{b}_{a_0}$;

(3) $\sup_{1\leq a\leq b}diam(B^{b}_a)\to 0$ as $b\to \infty$.
\end{definition}

For any $B_a^b\in \mathbb{P}_b$, we define measures $\mu^{B_a^b}_n$ as in (\ref{definition of original measure in l.u.i}) by replacing $B$ with $B_a^b$. For any fixed $b$ and $a$, After passing to a subsequence, we assume that $\mu^{B_a^b}_n$ converges to a probability measure $\mu^{B_a^b}$ on $Y$ as $n\to \infty$.

\begin{definition}\label{definition of locally unipotent invariant measure}
The probability measure $\mu^B$ is \textbf{locally unipotent invariant} if there exist a compact subset $X\prm$ of $B$ with $\lambda(X\prm)=0$, a map $Z:B \to G$ which is continuous on $B\backslash X\prm$ and an admissible sequence of finite partitions $\{\mathbb{P}_b=\{B^b_a:1\leq a\leq b\}\}_{b\in \mathbb{N}}$ of $B$ such that the following hold:

(1) For any $b\in \mathbb{N}$, for any $B_a^b\in \mathbb{P}_b$ such that $B_a^b\cap X\prm=\emptyset$, the measure $\nu_n^{B^b_a}$ defined by 
\begin{align}\label{definition of perturbed measures in l.u.i}
    \nu_n^{B_a^b}:=\frac{1}{\lambda(B^b_a)}\int_{B^b_a} \delta_{Z(y)\theta_n(y)\Gamma/\Gamma}d\lambda(y)
\end{align}
converges to a probability measure $\nu^{B^b_a}$ on $G/\Gamma$ as $n\to \infty$;

(2) There exists a nontrivial one-parameter unipotent subgroup $\rho:\bR\to G$ such that for any $b\in \mathbb{N}$, for any $B_a^b\in \mathbb{P}_b$ with $B_a^b\cap X\prm=\emptyset$, $\nu^{B^b_a}$ is invariant under $\rho$.
\end{definition}

\begin{remark}
1. By definition above, if $\mu^B$ is invariant under a nontrivial unipotent subgroup, then $\mu^B$ is automatically locally unipotent invariant.

2. We allow the continuous map $Z$ to be undefined in the compact subset $X\prm$ with zero $\lambda$ measure. In typical examples, $Z(y)$ leaves any compact subset of $G$ when $y$ approaches $X\prm$. Since we are working with weak-* topology of measures, this singular set $X\prm$ does no harm.
\end{remark}

\begin{example}
For $d\geq 2$, let $G=SL_d(\bR)$ and $\Gamma=SL_d(\mathbb{Z})$.

Let $a_t=diag(e^{(d-1)t},e^{-t},\cdots,e^{-t})$. Assume that $\varphi: [0,1]\to \bR^{d-1}$ is an analytic map. The expanding horosphere of $a_t$ in $G$ can be identified with $\bR^{d-1}$. Let $u(\varphi(s)))$ be the image of $\varphi(s)$ through this identification. Let $t_n\to \infty$ be an increasing sequence. Consider the measure $\mu_n$ defined by 
\begin{align*}
    \mu_n=\int_0^1 \delta_{a_{t_n}u(\varphi(s))\Gamma/\Gamma}ds.
\end{align*}
It is shown in \cite{shah2009equidistribution} that $\mu_n$ converges to a locally unipotent invariant probability measure $\mu$ on $G/\Gamma$.
\end{example}

\begin{example}
Let $J$ be a nontrivial box contained in $[0,1]^k$ with positive Lebesgue measure. Define probability measure $\mu_n^J$ as in (\ref{measure over subbox}). By Corollary \ref{cor nondivergence over subbox}, there exists an increasing subsequence $\{n_i\}_{i\in \mathbb{N}}$ such that $\mu_{n_i}^J\to \mu^J$ for $i\to \infty$. We will prove that $\mu^J$ is a locally unipotent invariant probability measure on $G/\Gamma$. 
\end{example}

We have the following property for locally unipotent invariant measure $\mu^B$:

\begin{lemma}\label{lemma: relation between original measure and perturbed measure}
Let $\mu^B$, $\{\mathbb{P}_b\}_{b\in\mathbb{N}}$ and the map $Z$ be as in Definition \ref{definition of locally unipotent invariant measure}. Assume that there exists a subgroup $H$ of $G$ such that for any $b\in \mathbb{N}$, any $B^b_a\in \mathbb{P}_b$ with $B^b_a\cap X\prm=\emptyset$, the limit measure $\nu^{B_a^b}$ is invariant under $H$. Furthermore, assume that $Im(Z)\subset N_G(H)$, where $N_G(H)$ is the normalizer of $H$ in $G$. Then $\mu^B$ is invariant under $H$.
\end{lemma}

\begin{proof}
By definition of $\mu^B$, for any $b\in \mathbb{N}$, we have \begin{align*}
    \mu^B=\frac{1}{\lambda(B)}\sum_{a=1}^b \lambda(B^b_a)\mu^{B^b_a}.
\end{align*}
Since $\lambda(X\prm)=0$ and $X\prm$ is compact, it suffices to show that $\mu^{B^b_a}$ is $H$-invariant for any $B^b_a$ with $B^b_a\cap X\prm=\emptyset$. For the following argument we fix a $B_{a_0}^{b_0}$ and replace $B_{a_0}^{b_0}$ by $B$, assuming that $B\cap X\prm=\emptyset$.

Fix a function $\varphi \in C_c(X)$. For any $g\in G$, define a function $\varphi_g(x)=\varphi(g x)$ for any $x\in X$. Now fix an $h_0\in H$. Given $\epsilon>0$, depending on $\varphi$, $h_0$ and $\epsilon$, we can find a large enough $b\in \mathbb{N}$ such that any $1\leq a \leq b$, $\forall y,y\prm \in B^b_a$,
\begin{equation*}
    |\varphi(Z(y\prm)^{-1}Z(y)\theta_n(y)\Gamma)-\varphi(\theta_n(y)\Gamma)|<\epsilon ,
\end{equation*}
and
\begin{equation*}
    |\varphi_{h_0}(Z(y\prm)^{-1}Z(y)\theta_n(y)\Gamma)-\varphi_{h_0}(\theta_n(y)\Gamma)|<\epsilon.
\end{equation*}

To simply our notations, we write $B_a=B^b_a$ for $1\leq a\leq b$.
After passing to a subsequence, by definition of $\nu_n^{B_s}$,
\begin{align*}
    \nu^B=\frac{1}{\lambda(B)}\sum_{a=1}^b \lambda(B_a)\nu^{B_a}.
\end{align*}

For $1\leq a \leq b$, fix a point $y^a\in B_a$. Since $Im(Z)\subset N_G(H)$, $h_0\cdot Z(y^a)^{-1}=Z(y^a)^{-1}\cdot h_a$ for some $h_a\in H$. Choose $N\in \mathbb{N}$ large enough such that for all $n\geq N$ and $1\leq a \leq b$, we have the following estimates, where for two real numbers $A$ and $B$, $A \overset{\epsilon}{\sim} B$ means $|A-B|<\epsilon$:

\begin{equation*}\label{inequality 4 in equidistribution}
     \frac{1}{\lambda(B_a)}\int_{B_a}\varphi(\theta_n(y)\Gamma)d\lambda(y)\overset{\epsilon}{\sim}\int_X \varphi d\mu^{B_a};
\end{equation*}

\begin{equation*}\label{inequality 1 in equidistribution} 
    \int_X\varphi_{h_0}d\mu^{B_a} \overset{\epsilon}{\sim}\frac{1}{\lambda(B_a)}\int_{B_a}\varphi_{h_0}(\theta_n(y)\Gamma)d\lambda(y);
\end{equation*}

\begin{equation}\label{inequality 2 in equidistribution}
    \frac{1}{\lambda(B_a)}\int_{B_a}\varphi_{Z(y^a)^{-1}}(h_a Z(y)\theta_n(y)\Gamma)d\lambda(y)\overset{\epsilon}{\sim}\int_X\varphi_{Z(y^a)^{-1}}d\nu^{B_a};
\end{equation}

\begin{equation*}
    \int_X\varphi_{Z(y^a)^{-1}}d\nu^{B_a}\overset{\epsilon}{\sim}\frac{1}{\lambda(B_a)}\int_{B_a}\varphi_{Z(y^a)^{-1}}(Z(y)\theta_n(y)\Gamma)d\lambda(y).
\end{equation*}

where (\ref{inequality 2 in equidistribution}) follows by $H$-invariance of $\nu^{B_a}$. Therefore, for all $n\geq N$, we have the following estimates:
\begin{align*}
    \int_X \varphi_{h_0}d\mu^{B_a} &\overset{\epsilon}{\sim} \frac{1}{\lambda(B_a)}\int_{B_a} \varphi_{h_0}(\theta_n(y)\Gamma)d\lambda(y)\\
    &\overset{\epsilon}{\sim} \frac{1}{\lambda(B_a)}\int_{B_a} \varphi_{h_0}(Z(y^a)^{-1}Z(y)\theta_n(y)\Gamma)d\lambda(y)\\&=\frac{1}{\lambda(B_a)}\int_{B_a} \varphi(Z(y^a)^{-1}h_a Z(y)\theta_n(y)\Gamma)d\lambda(y) \\
   &=\frac{1}{\lambda(B_a)}\int_{B_a} \varphi_{Z(y^a)^{-1}}(h_aZ(y)\theta_n(y)\Gamma)d\lambda(y)\\&\overset{\epsilon}{\sim}\int_X\varphi_{Z(y^a)^{-1}}d\nu^{B_a}\\
   &\overset{\epsilon}{\sim}\frac{1}{\lambda(B_a)}\int_{B_a} \varphi_{Z(y^a)^{-1}}(Z(y)\theta_n(y)\Gamma)d\lambda(y)\\
    &\overset{\epsilon}{\sim} \frac{1}{\lambda(B_a)}\int_{B_a} \varphi(\theta_n(y)\Gamma)d\lambda(y)\\
    &\overset{\epsilon}{\sim} \int_X \varphi d\mu^{B_a}.
\end{align*}
Then since
\begin{align*}
    \int_X \varphi d\mu^B=\frac{1}{\lambda(B)}\sum_{a=1}^b \lambda(B_a)\int_X \varphi d\mu^{B_a},
\end{align*}
we have 
\begin{align*}
    \left| \int_X \varphi_{h_0} d\mu^B-\int_X\varphi d\mu^{B} \right|&\leq\frac{1}{\lambda(B)}\sum_{a=1}^b \lambda(B_a)\left| \int_X(\varphi_{h_0}-\varphi)d\mu^{B_a} \right|\\
    &\leq \frac{1}{\lambda(B)}\sum_{a=1}^b \lambda(B_a)6\epsilon =6\epsilon.
\end{align*}
Since $\epsilon$, $h_0$ and $\varphi$ are arbitrary, $\mu^B$ is $H$-invariant.
\end{proof}

Now we will prove that for any closed subbox $J\subset [0,1]^k$ with positive measure, the limit measure $\mu^J$ of measures $\mu^J_n$ defined as in (\ref{measure over subbox}) is locally unipotent invariant. This relies on the following technical lemma, for which we defer the proof to section \ref{Section Invariance under a unipotent flow}.

\begin{lemma}\label{local unipotent invariance}
Let $G$ be a real algebraic group and $\Theta:\bR^k \to G$ be a nonconstant regular algebraic map such that $\Theta(\bs{0})=e$. Given a sequence of boxes $\bx_n=[0,T_n^{\lambda_1}]\times\cdots\times[0,T_n^{\lambda_k}]$ such that $T_n\to \infty$ as $n\to \infty$ and $\lambda_i>0 $ for $1\leq i\leq k$. Let $\theta$ be defined as in (\ref{definition of rescaling map}). Then there exists $q\geq 0$ such that the following holds: there exist finitely many proper subvarieties $X_1,\cdots,X_m$ of $\bR^k$, for any $\valpha\in [0,1]^k\setminus \bigcup_{i=1}^m X_i$, there exists a nontrivial algebraic unipotent one-parameter subgroup $\rho_{\valpha}:\bR \to G$ such that for any $s\in \bR$,
\begin{align*}
    \lim_{t\to \infty}\theta(\valpha,t+st^{-q})\theta(\valpha,t)^{-1}=\rho_{\valpha}(s).
\end{align*}
\end{lemma}

Since $\rho_{\valpha}$ is a one-parameter algebraic unipotent subgroup, there is a nilpotent matrix $Y_{\valpha}$ such that
\begin{align}\label{exp of nilpotent element}
   \rho_{\valpha}(s)=\exp{(sY_{\valpha})} 
\end{align}
By equation (\ref{equation 0}), taking derivative of both sides in (\ref{exp of nilpotent element}) with respect to $s$ then set $s=0$, we have $Y_{\valpha}=M_1(\valpha)$. Therefore each matrix coordinate of $Y_{\valpha}$ is a polynomial in $\valpha$. Moreover, $Y_{\valpha}$ is nontrivial as $\valpha$ varies in $[0,1]^k\setminus \bigcup_{i=1}^m X_i$. 

We shall need the following fact for semisimple algebraic group defined over $\mathbb{R}$.

\begin{lemma}(\cite[Lemma 4.2]{yang2020equidistribution})\label{finite conjugacy classes}
Let $G$ be a semisimple real Lie group defined over $\mathbb{R}$, then there are only finitely many $G$-conjugacy classes of the nilpotent elements in the Lie algebra $\mathfrak{g}$ of $G$.
\end{lemma}

\begin{lemma}\label{existence of map Z}
Suppose that the image of $\Theta$ is contained in a semisimple real Lie group $G$ defined over $\mathbb{R}$. Then there exists a closed $X\prm\subset [0,1]^k$ with $|X\prm|=0$ such that for any $\valpha^0 \in [0,1]^k\setminus X\prm$, there exists a continuous map $Z:[0,1]^k\setminus X\prm\to G$ and
\begin{align}
    Ad(Z(\valpha)) Y_{\valpha}= Y_{\valpha^0}, \forall \valpha\in [0,1]^k\setminus X\prm.
\end{align}
\end{lemma}

\begin{proof}
By above discussion, $Y_{\valpha}$ is a nilpotent element for any $\valpha\in [0,1]^k\setminus \bigcup_{i=1}^m X_i$.  By Lemma \ref{finite conjugacy classes} and the fact that $\valpha\mapsto Y_{\valpha}$ is a polynomial map, for any $\valpha$ not belong to finitely many proper subvarieties $X_{m+1},\cdots,X_{r}$ in $\bR^k$, $Y_{\valpha}$ is in a single nontrivial $G$-conjugacy class of nilpotent elements. We can set 
\begin{align*}
    X\prm=\bigcup_{i=1}^r X_i.
\end{align*}
\end{proof}

Now we take a nontrivial box $J\subset [0,1]^k\setminus X\prm$. Given a sequence $T_n\to \infty$, we define $B_n^J$ as in (\ref{expanding subbox}). We also define a measure $\nu_{n}^J=\nu_{B_n}^J$ on $G/\Gamma$ by 
\begin{align}\label{modified measure over subbox}
    \int_{G/\Gamma} f d\nu_{n}^J:=\frac{1}{|J|}\int_{\valpha\in J}f(Z(\valpha)\theta(\valpha,T_n)\Gamma)d\valpha, \forall f\in C_c(G/\Gamma).
\end{align}

\begin{lemma}\label{lemma nondivergence of modified measure}
After passing to a subsequence of $\{T_n\}_{n\in \mathbb{N}}$, the measures $\nu_n^J$ defined in (\ref{modified measure over subbox}) converge to $\nu^J$ as $n\to \infty$, where $\nu^J$ is a probability measure on $G/\Gamma$.
\end{lemma}

\begin{proof}
Since the map $Z:J \to G$ is continuous and $J$ is compact, $Z(J)$ is compact. By Corollary \ref{cor nondivergence over subbox}, for any $\epsilon>0$, there is a compact subset $K$ of $G/\Gamma$ such that $\mu^J_n(K)>1-\epsilon$ for all $n$ large enough, where $\mu_n^J$ is defined in (\ref{measure over subbox}). Let $K\prm=Z(J)K$ be a compact set, then by definition of $\nu_n^J$, $\nu_n^J(K\prm)>1-\epsilon$ for all $n$ large enough. Therefore,
after passing to a subsequence of $\{T_n\}_{n\in \mathbb{N}}$, $\nu_{n}^J\to \nu^J$ as $n\to \infty$ where $\nu^J$ is a probability measure on $G/\Gamma$. 
\end{proof}

\begin{lemma}\label{measure is unipotent invariance after twisted }
Fix $\valpha^0\in [0,1]^k\setminus X\prm$ and $J=[\beta_1,\beta_1+\delta_1]\times\cdots\times[\beta_k,\beta_k+\delta_k]\subset [0,1]^k\setminus X\prm$. The measure $\nu^J$, which is the limit measure of $\nu^J_n$ as defined in (\ref{modified measure over subbox}), is invariant under the one-parameter unipotent group $\{\exp (s Y_{\valpha^0}):s\in \bR\}$, where $Y_{\valpha^0}$ is defined in Remark \ref{existence of map Z}.
\end{lemma}
\begin{proof}
Let $f\in C_c(G/\Gamma)$, fix $s\in \bR$, we have the following:
\begin{align*}
    \int_{X}f(\exp (sY_{\valpha^0}) x)d\nu^J(x)&=\lim_{n\to \infty} \frac{1}{|J|}\int_{J}f(\exp (sY_{\valpha^0})Z(\valpha)\theta(\valpha,T_n)\Gamma)d\valpha\\
    &=\lim_{n\to\infty}\frac{1}{|J|}\int_{J}f(Z(\valpha)\theta(\valpha,T_n+s T_n^{-q})\Gamma)d\valpha \numberthis \label{equation 4} 
\end{align*}

Where $q\geq 0$ is in Lemma \ref{local unipotent invariance}. Note that (\ref{equation 4}) follows by continuity of map $Z$, (\ref{equation 2}) and the identity $\exp (s Y_{\valpha^0}) Z(\valpha)=Z(\valpha)\exp( s Y_{\valpha})$. Make change of variables by setting $S_n=T_n+s T_n^{-q}$, and $x_i=\alpha_i S_n^{\lambda_i}$ for $i=1,\cdots,k$, we have $Z(\valpha(\vx))=Z(x_1/S_n^{\lambda_1},\cdots,x_k/S_n^{\lambda_k})$ and (\ref{equation 4}) equals
\begin{align*}
    &\lim_{n\to\infty}\frac{1}{\prod_{i=1}^{k}\delta_i S_n^{\lambda_i}} \int_{\beta_1 S_n^{\lambda_1}}^{(\beta_1+\delta_1)S_n^{\lambda_1}}\cdots\int_{\beta_k S_n^{\lambda_k}}^{(\beta_k+\delta_k)S_n^{\lambda_k}}f(Z(\valpha(\vx))\Theta(\vx)\Gamma)d\vx=\\
    &\lim_{n\to\infty}\frac{1}{\prod_{i=1}^{k}\delta_i T_n^{\lambda_i}} \int_{\beta_1 T_n^{\lambda_1}}^{(\beta_1+\delta_1)T_n^{\lambda_1}}\cdots\int_{\beta_k T_n^{\lambda_k}}^{(\beta_k+\delta_k)T_n^{\lambda_k}}f(Z(\valpha(\vx))\Theta(\vx)\Gamma)d\vx \numberthis \label{equation 5}=\\
    &\lim_{n\to\infty}\frac{1}{\prod_{i=1}^{k}\delta_i T_n^{\lambda_i}} \int_{\beta_1 T_n^{\lambda_1}}^{(\beta_1+\delta_1)T_n^{\lambda_1}}\cdots\int_{\beta_k T_n^{\lambda_k}}^{(\beta_k+\delta_k)T_n^{\lambda_k}}f(Z(\valpha\prm(\vx))\Theta(\vx)\Gamma)d\vx \numberthis \label{equation 6},\\
\end{align*}
where $Z(\valpha\prm(\vx))=Z(x_1/T_n^{\lambda_1},\cdots,x_k/T_n^{\lambda_k})$.

As above, (\ref{equation 5}) and (\ref{equation 6}) follows from the facts that $T_n/S_n \to 1$ as $n\to \infty$, uniform continuity of $f$ and continuity of $Z$. Make change of variables again by setting $\alpha_i=x_i/T_n^{\lambda_i}$ for $i=1,\cdots,k$, (\ref{equation 6}) equals 
\begin{align*}
    \lim_{n\to \infty}\frac{1}{|J|}\int_J f(Z(\valpha)\theta(\valpha,T_n)\Gamma)d\valpha=\int_X f(x)d \nu^J(x),
\end{align*}
which completes the proof.
\end{proof}

\begin{proposition}\label{mu is locally unipotent invariant}
For any nontrivial box $J\subset [0,1]^k$, the probability measure $\mu^J$ on $G/\Gamma$ is locally unipotent invariant.
\end{proposition}

\begin{proof}
As in Definition \ref{definition of locally unipotent invariant measure}, we can take $\theta_n(\cdot)=\theta(\cdot,T_n)$ and $B=J$. We can choose $\mathbb{P}_b$ to be the finite partition of $J$ consisting of nontrivial subboxes with size less than $1/b$. $X\prm$ and the map $Z$ can be taken to be the ones in Lemma \ref{measure is unipotent invariance after twisted }. Condition (1) in Definition \ref{definition of locally unipotent invariant measure} is guaranteed by Lemma \ref{lemma nondivergence of modified measure}. Condition (2) is satisfied by Lemma \ref{measure is unipotent invariance after twisted }.
\end{proof}

\section{Proof of main Theorem}\label{proof of main theorems}

Recall that in Theorem \ref{main theorem}, we assume that $G$ is a real algebraic group, and there are closed subgroups $\Gamma\subset G_1 \subset G$ such that $G_1/\Gamma$ admits a $G_1$-invariant probability measure. Let $\Theta:\bR^k \rightarrow G_1$ be a nontrivial regular algebraic map such that $\Theta(\boldsymbol{0})=e$. Let $F$ be the smallest closed connected subgroup of $G_1$ containing $\Theta(\bR^k)$ such that $F\Gamma/\Gamma$ is closed.

First of all, we make some reductions for the proof. By \cite[Proposition 2.1]{shah1994limit}, $F\Gamma/\Gamma$ admits a F-invariant probability measure. Moreover, by \cite[Note 2.1]{shah1994limit}, replacing $\Gamma$ by $F\cap \Gamma$, $G_1$ by $F$ and $G$ by Zariski closure of $F$, we can assume that there is no proper algebraic subgroup $A$ of $G$ such that $\Theta(\bR^k)\subset A$ and $A\Gamma/\Gamma$ is closed. Also, we can assume that $\Gamma^0$, the connected component of $\Gamma$, is normal in $G$.

By \cite[Note 2.2]{shah1994limit}, let $W$ be the closed subgroup generated by all algebraic unipotent one-parameter subgroups of $G$ contained in $\Gamma^0$, then $W$ is a normal real algebraic subgroup of $G$. Let $q:G\to G/W$ be the natural quotient map, then $G/\Gamma \cong q(G)/q(\Gamma)$ $G$-equivalently. So we can assume that $\Gamma$ contains no nontrivial algebraic unipotent one-parameter subgroups of $G$. Also, we note that if $G$ is a semisimple real algebraic group, so is $G/W$.

Given a sequence $T_n \to\infty$, positive numbers $\lambda_i$ for $i=1,\cdots,k$ and a regular algebraic map $\Theta:\bR^k\to G$. Recall that we define a map $\theta:[0,1]^k \times (0,\infty)\to G$ by $\theta(\valpha,t)=\Theta(\valpha_1 t^{\lambda_1},\cdots,\valpha_k t^{\lambda_k})$. For a subbox $J\subset [0,1]^k\setminus X\prm$ where $X\prm$ is as in lemma \ref{existence of map Z}, and any $n\in\mathbb{N}$, we define probability measure $\nu_n^J$ as in (\ref{modified measure over subbox}).

\begin{theorem}\label{equadistribution over subbox}

Let $G$ be a real algebraic semisimple group. Let $\Gamma\subset G$ be a closed subgroup such that $G/\Gamma$ admits a finite $G$-invariant measure. Let $\Theta:\bR^k \to G$ be a nonconstant regular algebraic map and $\Theta(\bs{0})=e$. Assume that there is no proper closed subgroup $F$ of $G$ containing $\Theta(\mathbb{R}^k)$ such that $F\Gamma/\Gamma$ is closed. Then given any sequence $T_n \to \infty$ as $n\to \infty$ and $\lambda_i>0$ for $i=1,\cdots,k$, for any subbox $J\subset [0,1]^k\setminus X\prm$, the probability measure $\nu^J_n$ defined in (\ref{modified measure over subbox}) converges to the Haar measure $\mu_X$ on $X=G/\Gamma$, as $n\to \infty$.  
\end{theorem}

\begin{proof}
By the reduction at the beginning of this section, we can assume that the connected component $\Gamma^0$ of $\Gamma$ is normal in $G$ and $\Gamma$ contains no nontrivial one-parameter unipotent subgroups.

We proceed by induction on $\dim(G/\Gamma)$. If $\dim(G/\Gamma)=0$, then the theorem is trivial. Assume that the theorem has been proved for any closed subgroup $\Lambda\subset G$ such that $\Lambda^0$ is normal in $G$ and $G/\Lambda$ has a $G$-invariant probability measure with $dim(G/\Lambda)\leq m-1$. Now we want to prove the theorem when $dim(G/\Gamma)=m$.

By Lemma \ref{lemma nondivergence of modified measure} and Lemma \ref{measure is unipotent invariance after twisted }, after passing to a subsequence, $\nu^J_n$ converges to $\nu^J$, which is a probability measure on $G/\Gamma$ invariant under a nontrivial unipotent subgroup. Let $W$ be the closed connected subgroup generated by those unipotent subgroups preserving $\nu^J$, then $W$ is nontrivial. By Theorem \ref{Ratner's Theorem}, there exists $H\in \mathcal{H}$ such that
\begin{align}\label{smallest tube has positive measure}
    \nu^{J}(\pi(N(H,W)))>0 \text{ and  } \nu^J(\pi(S(H,W)))=0.
\end{align}

Let $C$ be a compact subset of $N(H,W)\setminus S(H,W)$ such that there exists $\epsilon>0$ with $\nu^{J}(\pi(C))>\epsilon$. Let $K$ be a compact neighborhood of $\pi(C)$ in $X$, then by Corollary \ref{modification of linearization}, let $H_0=H$, we can find $H_1,\cdots,H_r \in \mathcal{H}$ with $H_i\subset H$, $\dim H_i <\dim H$ and compact sets $D_i\subset A_{H_i}$ for $i=0,\cdots, r$ such that the following holds: For any neighborhood $\Phi_i$ of $D_i$ in $V_{H_i}$ for $i=0,\cdots,r$, we can find a small compact neighborhood $\mathcal{U}$ of $e$ in $G$ (depends only on $C$ and $K$), a neighborhood $\tilde{\Omega}$ of $\pi(C)$ and a compact neighborhood $C\subset \tilde{K}\subset K$ such that for any box $\bx\subset [0,1]^k$, any continuous map $Z\prm:\bx \rightarrow \mathcal{U}$, any $\theta \in \classP$ and any $g\in G$, either

$(1\prm)$ $\exists \gamma \in \Gamma$, such that $Z\prm(\bx)\theta(\bx)g\gamma$ is contained in one of the components of $\eta^{-1}_{H_0}(\mathcal{U}\cdot\Phi_0)\cup \cdots \cup \eta^{-1}_{H_r}(\mathcal{U}\cdot\Phi_r)$, or

($2\prm$) $|\{\valpha \in \bx: Z\prm(\valpha)\theta(\valpha)g\Gamma\in \tilde{\Omega}\cap \tilde{K}\}|\leq \epsilon\prm |\bx|$.

To apply the above dichotomy, let $\mathbb{P}_b$ be an admissible finite partition of $J$ into nontrivial boxes $J_a^b$ for $1\leq a\leq b$, each with size less than $1/b$. By continuity of $Z$, if $b\in \mathbb{N}$ is large enough, then for any two points $\valpha, \valpha\prm$ in any one of such boxes, we have $Z(\valpha)Z(\valpha\prm)^{-1}\in \mathcal{U}$. Since $\nu^J(\pi(C))>\epsilon$, after passing to a subsequence of $\{T_n\}_{n\in\mathbb{N}}$, we can find a box $J_a^b=J_1$ in the partition such that for all $n$ large enough,
\begin{align}\label{(1) doesn't hold}
    \frac{1}{|J_1|}\left|\{\valpha \in J_1: Z(\valpha)\theta(\valpha,T_n)\Gamma\in \tilde{\Omega}\cap K\}\right|>\epsilon.
\end{align}
Now fix a point $\valpha^0 \in J_1$. For any $\valpha\in J_1$, we can write
\begin{align*}
    Z(\valpha)\theta(\valpha,T_n)=Z(\valpha)Z(\valpha^0)^{-1}\cdot Z(\valpha^0)\theta(\valpha,T_n).
\end{align*}

Note that by our construction, $Z(\valpha)Z(\valpha^0)^{-1}\in \mathcal{U}$. Also, since $Z(\valpha^0)$ is constant and $\theta(\valpha,T_n)\in \classP$ is a polynomial in $\valpha$ for all $n$, we have $Z(\valpha^0)\theta(\valpha,T_n)\in \classP$. For each $n$, we apply the dichotomy to  $Z(\valpha)Z(\valpha^0)^{-1}$, $Z(\valpha^0)\theta(\valpha,T_n)$, $J_1$ and $e$ in place of $Z\prm(\valpha)$, $\theta(\valpha)$, $\bx$ and $g$.

By the above discussion and (\ref{(1) doesn't hold}), condition $(1\prm)$ always holds for all $n$ large enough. After passing to a subsequence of $\{T_n\}_{n\in\mathbb{N}}$, there exists $0\leq r_0\leq r$ such that for any $n\in \mathbb{N}$, there exists $\gamma_n\in \Gamma$ such that $Z(\valpha)\theta(\valpha,T_n)\gamma_n\in \eta_{H_{r_0}}^{-1}(\mathcal{U}\cdot \Phi_{r_0})$ for all $\valpha\in J_1$. Since $\mathcal{U}\cdot \Phi_{r_0}$ is bounded, there is some $R>0$ such that for all $n$,
\begin{align*}
    \|Z(\valpha)\theta(\valpha,T_n)\gamma_n p_{H_{r_0}}\|<R, \forall \valpha\in J_1,
\end{align*}
where $\|\cdot\|$ is some fixed norm on the vector space. By compactness of $J_1$ and continuity of the map $Z$, we have for all $n$,
\begin{align}\label{sup on some subbox}
    \|\theta(\valpha,T_n)\gamma_n p_{H_{r_0}}\|<R\prm, \forall \valpha\in J_1,
\end{align}
for some $R\prm \geq R$.

For each fixed $n$, $\|\theta(\valpha,T_n)\gamma_n p_{H_{r_0}}\|$ can be identified as a polynomial in $\valpha$ of degree less than $l$ with coefficients being functions of $T_n$. Then by Lemma \ref{consequence of good} and (\ref{sup on some subbox}), we can find a finite $R^{\prime\prime}>0$ depending on $R\prm$ and $J_1$ such that for all $n\in \mathbb{N}$,
\begin{align*}
    \|\theta(\valpha,T_n)\gamma_n p_{H_{r_0}}\|<R^{\prime\prime}, \forall \valpha\in [0,1]^k.
\end{align*}
Take $\valpha=\boldsymbol{0}$, we have $\|\gamma_n p_{H_{r_0}}\|<R^{\prime\prime}$.

By discreteness of $\Gamma\cdot p_{H_{r_0}}$ (see Theorem \ref{discrete orbit}), the number of $\gamma\in \Gamma$ such that $\|\gamma p_{H_{r_0}}\|< R^{\prime\prime}$ is finite. After passing to a subsequence of $\{T_n\}_{n\in \mathbb{N}}$, we can find $\gamma_0\in \Gamma$ such that for all $n$,
\begin{align*}
    \|\theta(\valpha,T_n)\gamma_0 p_{H_{r_0}}\|<R^{\prime\prime},\forall \valpha \in [0,1]^k.
\end{align*}
By definition of $\theta(\valpha,T_n)$, this says that
\begin{align}\label{bounded image}
    \|\Theta(\bR^k)\gamma_0 p_{H_{r_0}}\|< R^{\prime\prime}.
\end{align}

Since the map $\vx\mapsto \|\Theta(\vx)\gamma_0 p_{H_{r_0}}\|^2$ is a polynomial, $\|\Theta(\vx)\gamma_0 p_{H_{r_0}}\|^2$ is either constant or unbounded as $\vx\to \infty$. Hence by (\ref{bounded image}), we have 
$\Theta(\bR^k)\gamma_0 p_{H_{r_0}}=\gamma_0 p_{H_{r_0}}$.

Therefore, $\Theta(\bR^k)\subset \gamma_0 N^1(H_{r_0})\gamma_0^{-1}$ and $\Theta(\bR^k)\Gamma/\Gamma\subset \gamma_0 N^1(H_{r_0})\Gamma/\Gamma$. Since $N^1(H_{r_0})\Gamma/\Gamma$ is closed by Theorem \ref{discrete orbit}, $\Gamma^0\subset N^1(H_{r_0})$ and $N^1(H_{r_0})$ is algebraic, by the assumption in the theorem, we have $N^1(H_{r_0})=G$. Thus $H_{r_0}$ is normal in $G$. Recall that in Corollary \ref{modification of linearization}, $H_{r_0}$ is chosen such that $N(H_{r_0},W)\neq \emptyset$. Now since $H_{r_0}$ is normal, $N(H_{r_0},W)=G$.  We conclude that $H_{r_0}=H$ by (\ref{smallest tube has positive measure}) and the fact that $N(H_i,W)\subset S(H,W)$ for $1\leq i\leq r$. Therefore $H$ is normal in $G$ and $\nu^J$ is $H$-invariant by (\ref{smallest tube has positive measure}) and Theorem \ref{Ratner's Theorem}.

Since $W\subset H$, $H$ contains a nontrivial unipotent subgroup. Let $\Lambda=H\Gamma$. By the reduction made at the beginning of this section, $\Gamma^0$ contains no nontrivial one-parameter unipotent subgroup, so $\dim\Lambda^0=\dim H\Gamma^0>\dim \Gamma^0$ and $\dim (G/\Gamma)>\dim(G/H\Gamma)$. Note that $G$ is the smallest subgroup $F$ of $G$ such that $\Theta(\bR^k)\subset F$ and $F\Lambda/\Lambda$ is closed. Since $H$ is normal, $\Lambda^0=H\Gamma^0$ is normal in $G$. Let $q:G/\Gamma\to G/\Lambda$ be the natural quotient map. Then by induction hypothesis, $q_*(\nu^J)$ is $G$-invariant. Since the fibers of $q$ are closed $H$-orbits and $\nu^J$ is $H$-invariant, $\nu^J$ is also $G$-invariant by \cite[Proposition 1.6]{dani1978invariant} and we finish the proof.
\end{proof}

\begin{proof}[proof of Theorem \ref{main theorem}]
Denote $B=[0,1]^k$, we need to show that $\mu^B$ is $G$-invariant.
By Theorem \ref{equadistribution over subbox}, $\nu^J$ is $G$-invariant for any box $J\subset [0,1]^k\setminus X\prm$. Since $\mu^B$ is locally unipotent invariant by Proposition \ref{mu is locally unipotent invariant}, by Lemma \ref{lemma: relation between original measure and perturbed measure} we conclude that $\mu^B$ is $G$-invariant.

\end{proof}

\section{Locally Unipotent invariance of the limit measure}\label{Section Invariance under a unipotent flow}\label{Invariance under a unipotent flow}

\begin{proof}[proof of Lemma \ref{local unipotent invariance}]
Let $M(N,\bR)$ denote the affine space of $N\times N$ real matrices which contains $G$ as an affine subvariety. There exist polynomials $\theta_{ij}(\valpha,t)$ in variable $\valpha$ for $i,j=1,\cdots,N$ such that $\theta(\valpha,t)=(\theta_{ij}(\valpha,t))_{N\times N}$. Since $\Theta$ is nonconstant, the set of $\valpha \in [0,1]^k$ where $\theta(\valpha,t)$ as a function of $t$ is constant is a proper subvariety in $[0,1]^k$. We treat $\valpha$ in $\theta(\valpha,t)$ as an indeterminant variable for now.

First we find $q>0$ as in the lemma by induction. To simplify our proof, without loss of generality, we rescale the parameter $T_n$ to ensure that $\lambda_i>1$ for any $1\leq i \leq k$ and that at least one power of $t$ in some matrix coordinate of $\theta(\valpha,t)$ is not an integer. Let $\theta^{(l)}(\valpha,t)$ denote the $l$-th matrix coordinatewise derivative of the matrix $\theta(\valpha,t)$ with respect to $t$. For a matrix $(M_{ij}(t))_{N\times N}\in M(N,\bR)$ each of whose matrix coordinate is a linear combination of real powers of $t$, let $\deg(M(t))=\max_{1\leq i,j\leq N}(\deg(M_{ij}(t)))$ (for example, $\deg(t^{p_1}+2t^{p_2})=p_1$, for real numbers $p_1>p_2$). By convention, we take the degree of zero constant function to be $-\infty$. Put
\begin{align*}
    d_1:=\min\{l\in \mathbb{N}: -\infty<\deg\theta^{(l+1)}(\valpha,t)<0\}.
\end{align*}

Note that $d_1<\infty$ by the above assumption. Since the map $G\ni g\mapsto g^{-1}$ is an algebraic morphism on $G$, the map $\bR^k \ni \vx \mapsto \Theta(\vx)^{-1}$ is also regular algebraic. Therefore, $\deg\theta(\valpha,t)^{-1}\geq 0$. Let 
\begin{align*}
    q_1:=\max_{1\leq l \leq d_1}\frac{1}{l}\deg(\theta^{(l)}(\valpha,t)\theta(\valpha,t)^{-1}).
\end{align*}

Note that $q_1\geq 0$ since $\deg(\theta^{(1)}(\valpha,t)\theta(\valpha,t)^{-1})\geq 0$ by assumption on $\lambda_1,\cdots,\lambda_k$. Indeed, $\theta^{(1)}(\valpha,t)\theta(\valpha,t)^{-1}$ has no negative power term of $t$ since $\lambda_i>1$ for all $i$.

Now we argue that $q_1$ must be positive. If not, then $q_1=0$, and
\begin{align*}
    \deg(\theta^{(1)}(\valpha,t)\theta(\valpha,t)^{-1})=0.
\end{align*}
So $\theta^{(1)}(\valpha,t)\theta(\valpha,t)^{-1}=A(\valpha)$ for some matrix $A(\valpha)$ only depends on $\valpha$, this is because there is no negative power term of $t$ and therefore $\theta^{(1)}(\valpha,t)\theta(\valpha,t)^{-1}$ must be a constant matrix of $t$. We can write
\begin{align}\label{contradiction 1}
    \theta^{(1)}(\valpha,t)=A(\valpha)\cdot\theta(\valpha,t).
\end{align}

Taking derivative on both sides of (\ref{contradiction 1}) for $t$, we have 
\begin{align}\label{contradiction 2}
    \forall l\geq 1, \theta^{(l)}(\valpha,t)=A(\valpha)^{l}\cdot \theta(\valpha,t).
\end{align}

Note that $-\infty<\deg(\theta^{(d_1+1)}(\valpha,t))<0$ since some power of $t$ in $\theta(\valpha,t)$ is not an integer. But $\deg(A(\valpha)^{d_1+1}\cdot \theta(\valpha,t))\geq 0$ or equals $-\infty$ (which happens when it is a zero matrix). Thus, (\ref{contradiction 2}) is a contradiction when $l=d_1$. Therefore $q_1>0$.

Now for $i\geq 1$, if there exists some constant $\xi$ with $|\xi|<1$ and $\deg(\theta^{(d_i+1)}(\valpha,t+\xi)\theta(\valpha,t)^{-1}t^{-q_i(d_i+1)})\geq 0$, we set 
\begin{align*}
    d_{i+1}:=d_i+1 \text{ and } q_{i+1}:=\max_{1\leq l\leq d_{i+1}}\frac{1}{l}\deg(\theta^{(l)}(\valpha,t)\theta(\valpha,t)^{-1}).
\end{align*}

This process must end in finitely many steps because $\deg(\theta^{(d_i+1)}(\valpha,t)\theta(\valpha,t)^{-1})$ is bounded above but $\deg(t^{-q_i(d_i+1)})=-q_i(d_i+1)$ is decreasing when $d_i$ increases, as $q_i\geq q_1>0$. 

Let $j$ be the smallest integer such that  $\deg(\theta^{(d_j+1)}(\valpha,t+\xi)\theta(\valpha,t)^{-1}t^{-q_j(d_j+1)})< 0$ for any constant $|\xi|<1$. Let $q=q_j$ and $d=d_j$. Now fix $s\in \bR$, for each matrix coordinate, using Taylor expansion for $t$, we have 
\begin{align}\label{taylor expansion 1}
    \theta(\valpha,t+st^{-q})=\theta(\valpha,t)+\sum_{l=1}^d \theta^{(l)}(\valpha,t)t^{-ql}\frac{s^l}{l!}+\theta^{(d+1)}(\valpha,c_t)t^{-q(d+1)}\frac{s^{d+1}}{(d+1)!}.
\end{align}

Where $c_t$ is a number between $t$ and $t+st^{-q}$ (indeed $c_t$ might be different for each matrix coordinate, to simplify our notation, we simply denote them by $c_t$). In particular, as $q>0$, $|c_t-t|<1$ as $t$ large enough. Multiply both sides of (\ref{taylor expansion 1}) on the right by $\theta(\valpha,t)^{-1}$, we have 
\begin{align*}
     \theta(\valpha,t+st^{-q})\theta(\valpha,t)^{-1}=&Id+\sum_{l=1}^d \theta^{(l)}(\valpha,t)\theta(\valpha,t)^{-1}t^{-ql}\frac{s^l}{l!}\\
     &+\theta^{(d+1)}(\valpha,c_t)\theta(\valpha,t)^{-1}t^{-q(d+1)}\frac{s^{d+1}}{(d+1)!}.
\end{align*}
By the choice of $q$, we have for all $1\leq l\leq d$,
\begin{align}\label{equation 1}
    \lim_{t\to \infty}\theta^{(l)}(\valpha,t)\theta(\valpha,t)^{-1}t^{-ql}=M_l(\valpha)\in M(n,\bR).
\end{align}

Note that each matrix coordinate of $M_l(\valpha)$ is a polynomial of $\valpha$. By the choice of $q$, the set of $\valpha$ where $(M_1(\valpha),\cdots,M_d(\valpha))=(0,\cdots,0)$ is a finite union of proper subvarieties of $\bR^k$. Let $X_1,\cdots, X_m$ be the proper subvarieties in $\bR^k$ such that $(M_1(\valpha),\cdots,M_d(\valpha))=(0,\cdots,0)$ if $\valpha\in \bigcup_{i=1}^m X_i$.

By the above construction, for any constant $\xi$ with $|\xi|<1$, we have $\deg(\theta^{(d+1)}(\valpha,t+\xi)\theta(\valpha,t)^{-1} t^{-q(d+1)})<0$. Take $\xi=c_t-t$. Note that $|c_t-t|<1$ as $t$ large enough. Thus,
\begin{align}\label{equation 1prm}
     \lim_{t\to \infty}\theta^{(d+1)}(\valpha,c_t)\theta(\valpha,t)^{-1}t^{-q(d+1)}=0. .
\end{align}

We put 
\begin{align}\label{equation 0}
    \rho_{\valpha}(s)=Id+\sum_{l=1}^d M_l(\valpha)\frac{s^l}{l!}
\end{align}
for all $s\in \bR$. Note that $\rho_{\valpha}(s)$ is nonconstant for $\valpha\in [0,1]^k\setminus\bigcup_{i=1}^m X_i$.

Now we prove that $\rho_{\valpha}$ is a nontrivial algebraic one-parameter unipotent subgroup. For any $s\in \bR$ and any map $t\mapsto s_t$ with $s_t\to s$ as $t\to \infty$, we have 
\begin{align}\label{equation 2}
    \lim_{t\to \infty}\theta(\valpha,t+s_t t^{-q})\theta(\valpha,t)^{-1}=\rho_{\valpha}(s).
\end{align}
Note that the above convergence is uniform for all $\valpha$ in any compact subset of $[0,1]^k\setminus \bigcup_{i=1}^m X_i$. Now for $s_1,s_2\in \bR$, we have 
\begin{align*}
    \rho_{\valpha}(s_1+s_2)\rho_{\valpha}(s_2)^{-1}&=\lim_{t\to \infty}\theta(\valpha,t+(s_1+s_2)t^{-q})\theta(\valpha,t)^{-1}\theta(\valpha,t)\theta(\valpha,t+s_2t^{-q})^{-1}\\
    &=\lim_{t\to \infty}\theta(\valpha,y_t+s_t y_t^{-q})\theta(\valpha,y_t)^{-1}\\
    &\text{where $y_t=t+s_2t^{-q}$ and $s_t=s_1(y_t/t)^q$}\\
    &=\rho_{\valpha}(s_1). \numberthis   \label{equation 3}
\end{align*}

By equations (\ref{equation 1}), (\ref{equation 2}) and (\ref{equation 3}), for $\valpha \in [0,1]^k\setminus \bigcup_{i=1}^m X_i$, $\rho_{\valpha}:\bR\to G$ is a nontrivial algebraic group homomorphism. Therefore $\rho_{\valpha}$ is a nontrivial algebraic unipotent one-parameter subgroup of $G$. This completes the proof.
\end{proof}

\section*{Acknowledgments} I would like to thank my thesis advisor, professor Nimish Shah, for suggesting this problem to me, for constant encouragements and generally sharing his ideas. 


\bibliographystyle{plain}
\bibliography{references.bib}

\medskip

\medskip

\end{document}